\documentclass[10pt]{article}
\usepackage{latexsym,amsmath,color,amsthm,amssymb,epsfig,graphicx,mathrsfs}
\usepackage{graphicx}
\usepackage{amssymb}
\usepackage[left=1in,top=1in,right=1in,bottom=1in]{geometry}
\usepackage[linktocpage=true]{hyperref}
\usepackage{setspace}

\usepackage{amssymb, amsmath, amsthm, graphicx,mathrsfs}

\def\qed{\hfill\ifhmode\unskip\nobreak\fi\quad\ifmmode\Box\else\hfill$\Box$\fi}
\def\ite#1{\hfill\break${}$\hbox to 50pt {\quad(#1)\hfill}}
\newtheorem{thm}{Theorem}[section]
\newtheorem{cor}[thm]{Corollary}

\newtheorem{rem}[thm]{Remark}
\newtheorem{lem}[thm]{Lemma}
\newtheorem{conj}[thm]{Conjecture}

\newtheorem{definition}[thm]{Definition}

\newtheorem{claim}[thm]{Claim}

\newtheorem{proper}{Property}
\begin{document}

\title{\vspace{-0.5in}
  On $r$-uniform hypergraphs with circumference less than $r$ }

\author{
{{Alexandr Kostochka}}\thanks{
\footnotesize {University of Illinois at Urbana--Champaign, Urbana, IL 61801
 and Sobolev Institute of Mathematics, Novosibirsk 630090, Russia. E-mail: \texttt {kostochk@math.uiuc.edu}.
 Research of this author
is supported in part by NSF grant DMS-1600592
and grants 18-01-00353A and 16-01-00499  of the Russian Foundation for Basic Research.
}}
\and{{Ruth Luo}}\thanks{University of Illinois at Urbana--Champaign, Urbana, IL 61801, USA. E-mail: {\tt ruthluo2@illinois.edu}.
Research of this author
is supported in part by Award RB17164 of the Research Board of the University of Illinois at Urbana-Champaign.
}}

\date{\today}

\maketitle

\vspace{-0.3in}

\begin{abstract}

We show that for each $k\geq 4$ and $n>r\geq k+1$, every $n$-vertex $r$-uniform hypergraph with
no Berge cycle of length at least $k$ has at most $\frac{(k-1)(n-1)}{r}$ edges. The bound is exact, and we 
describe the extremal hypergraphs. This implies and slightly refines the theorem of Gy\H ori, Katona and Lemons
that for $n>r\geq k\geq 3$, every $n$-vertex $r$-uniform hypergraph with
no Berge path of length  $k$ has at most $\frac{(k-1)n}{r+1}$ edges. To obtain the bounds, we study
bipartite graphs with no cycles of length at least $2k$, and then translate the results into the language
of multi-hypergraphs.

\end{abstract}

\section{Introduction}

%A cornerstone of extremal combinatorics is the study of Tur\'{a}n-type problems for graphs. 
The length $\ell(G)$ of a longest  path in  $G$ and the {\em  circumference} $c(G)$ (i.e. the length  of a longest cycle  in $G$)  are
fundamental parameters of a graph $G$. Erd\H{o}s and Gallai in 1959 proved the following results on these parameters.

\begin{thm}[Erd\H{o}s and Gallai~\cite{ErdGal59}]\label{ErdGallaiPath}  Let $n\geq k \geq 2$.
Let $G$ be an $n$-vertex graph with more than $\frac{1}{2}(k-2)n$ edges.
Then $G$ contains a $k$-vertex path $P_k$.
\end{thm}

\begin{thm}[Erd\H{o}s and Gallai~\cite{ErdGal59}]\label{ErdGallaiCyc}
Let $n\geq k \geq 3$. If $G$ is an $n$-vertex graph that does not contain a cycle of length at least $k$, then $e(G) \leq \frac{1}{2}(k-1)(n-1)$.  
%Let $G$ be an $n$-vertex graph with more than $\frac{1}{2}(k-1)(n-1)$ edges, $k \ge 3$.
%Then $G$ contains a cycle of length at least $k$.
\end{thm}

The bounds in Theorems~\ref{ErdGallaiPath} and~\ref{ErdGallaiCyc} are  best possible for infinitely many $n$ and $k$. The theorems were refined 
in~\cite{FaudSche75, FKV, Kopy, Lewin, Woodall}.
Recently, Gy\H{o}ri, Katona, and Lemons~\cite{lemons} and Davoodi, Gy\H{o}ri, Methuku, and Tompkins~\cite{davoodi}  extended Theorem~\ref{ErdGallaiPath} 
to $r$-uniform hypergraphs ({\em $r$-graphs}, for short). They considered  {\em Berge} paths and cycles. 

A {\em Berge-path} of length $k$ in a multi-hypergraph $\mathcal H$ is a set of $k$ hyperedges $\{e_1, \ldots, e_k\}$ and a set of $k+1$ base vertices $\{v_1, \ldots, v_{k+1}\}$ such that for each $1\leq i\leq k$, $v_i, v_{i+1} \in e_i$.

A {\em Berge-cycle} of length $k$ in a multi-hypergraph $\mathcal H$ is a set of $k$ hyperedges $\{e_1, \ldots, e_k\}$ and a set of $k$ base vertices $\{v_1, \ldots, v_k\}$ such that for each $i$, $v_i, v_{i+1} \in e_i$ (with indices modulo $k$). 

It turns out that the bounds behave differently in the cases $k\leq r$ and $k>r$.

\begin{thm}[Gy\H{o}ri, Katona and Lemons~\cite{lemons}]\label{EGp} Let $r \geq k \geq 3$, and let $\mathcal H$ be an $n$-vertex $r$-graph  with no Berge-path of length $k$. Then $e(\mathcal H) \leq \frac{(k-1)n}{r+1}$. 
\end{thm}

\begin{thm}[Gy\H{o}ri, Katona and Lemons~\cite{lemons}]\label{lemons} Let $k > r+1 > 3$, and let $\mathcal H$ be an  $n$-vertex $r$-graph  with no Berge-path of length $k$. Then $e(\mathcal H) \leq \frac{n}{k}{k \choose r}$. 
\end{thm}

Later, the remaining case $k = r+1$ was resolved by Davoodi, Gy\H{o}ri, Methuku, and Tompkins.

\begin{thm}[Davoodi, Gy\H{o}ri, Methuku and Tompkins~\cite{davoodi}] Let $k = r+1 > 2$, and let $\mathcal H$ be  an $n$-vertex
 $r$-graph  with no Berge-path of length $k$. Then $e(\mathcal H) \leq n$. 
\end{thm}

Furthermore, the bounds in these three theorems are sharp for each $k$ and $r$ for infinitely many $n$.

Very recently, several interesting results were obtained for Berge paths and cycles for $k\geq r+1$.
First,  Gy\H{o}ri, Methuku, Salia, Tompkins, and Vizer~\cite{connp} proved an asymptotic version of the Erd\H{o}s--Gallai theorem for Berge-paths in {\em connected} hypergraphs whenever $r$ is fixed and $n$ and $k$ tend to infinity.  
For Berge-cycles, the exact result for $k\geq r+3$ was obtained in~\cite{FKL}:
\begin{thm}[F\"uredi, Kostochka and Luo~\cite{FKL}]\label{FKL}
Let $k \geq r+3 \geq 6$, and let $\mathcal H$ be  an $n$-vertex $r$-graph   with no Berge-cycles of length $k$ or longer. Then $e(\mathcal H) \leq \frac{n-1}{k-2}{k-1 \choose r}$. 
\end{thm}
This theorem is a hypergraph version of Theorem~\ref{ErdGallaiCyc} and an analog of Theorem~\ref{lemons}. It also somewhat refines Theorem~\ref{lemons} for $k\geq r+3$.
Later,  Ergemlidze, Gy\H{o}ri, Methuku, Salia, Thompkins, and Zamora~\cite{EGMSTZ} extended the results to
to $k\in \{  r+1,r+2\}$:
\begin{thm}[Ergemlidze et al.~\cite{EGMSTZ}]\label{EGMSTZ}
If $k \geq 4$
 and  $\mathcal H$ is  an $n$-vertex $r$-graph  with no Berge-cycles of length $k$ or longer, then $k =r+1$ and $e(\mathcal H) \leq n-1$, or $k= r+2$ and $e(\mathcal H) \leq \frac{n-1}{k-2}{k-1 \choose r}$. 
\end{thm}

 %proved that the bound of Theorem~\ref{FKL} holds alsofor  $k\in \{  r+1,r+2\}$.
%.  Namely $e(\mathcal H) \leq n-1$ for $r= k+1$, and $e(\mathcal H) \leq \frac{n-1}{k-2}{k-1 \choose r}$ for $k=r+2$.

%\begin{thm}[Gy\H{o}ri, Methuku, Salia, Tompkins, Vizer~\cite{connp}]Let $\mathcal H$ be an $r$-uniform connected $n$-vertex hypergraph with no Berge-path of length $k$. Then \[\lim_{k\to \infty} \lim_{n \to \infty} \frac{|E(\mathcal H)|}{k^{r-1}n} \leq \frac{1}{2^{r-1}(r-1)!}.\]\end{thm}

The goal of this paper is to prove a hypergraph version of Theorem~\ref{ErdGallaiCyc}
 for $r$-graphs with no Berge-cycles of length $k$ or longer in the case $k\leq r$. Our result is an analog of Theorem~\ref{EGp} and  yields a refinement of it.
 
  Our approach is to consider  bipartite graphs in which the vertices in one of the parts have degrees at least $r$, and to 
analyze
the structure of such graphs with circumference less than $2k$. After that, we apply the obtained results to the incidence graphs
of $r$-uniform hypergraphs.
 In this way, our methods differ from those of \cite{lemons},\cite{davoodi}, \cite{connp}, and ~\cite{FKL}.

\section{Notation and results}

\subsection{Hypergraph notation}
The {\em lower rank} of a multi-hypergraph $\mathcal H$ is the size of a smallest edge of $\mathcal H$.

In view of the structure of our proof, it is more convenient to consider hypergraphs with lower rank at least $r$ instead of
$r$-uniform hypergraphs. It also yields formally stronger statements of the results.

The {\em incidence graph} $G(\mathcal H)$ of a  multi-hypergraph $\mathcal H=(V,E)$ is the bipartite graph with parts $V$ and $E$ where
 $v\in V$ is adjacent to $e\in E$ iff in $\mathcal H$ vertex $v$ belongs to edge $e$.

There are several versions of connectivity of hypergraphs.
We will call a multi-hypergraph $\mathcal H$  {\em  2-connected} if the incidence graph $G(\mathcal H)$ is 2-connected.

A {\em hyperblock} in a multi-hypergraph $\mathcal H$ is  a maximal  2-connected sub-multi-hypergraph of $\mathcal H$. 

%\end{document}

\begin{definition}For integers $r, k $ with $r \geq k+1$, we call a multi-hypergraph with lower rank at least $r$ an {\em $(r+1,k-1)$-block} if it contains exactly $r+1$ vertices and $k-1$ hyperedges.
\end{definition}

\begin{definition}A multi-hypergraph $\mathcal H$ with lower rank at least $r$ is an  {\em $(r+1,k-1)$-block-tree} if \\
(i) \;\; every hyperblock of $\mathcal H$ is an $(r+1,k-1)$-block,\\
(ii) \;\; all cut-vertices of the incidence graph $G(\mathcal H)$ of $\mathcal H$ are in $V$.
\end{definition}

\begin{figure}\centering
\includegraphics[scale=.5]{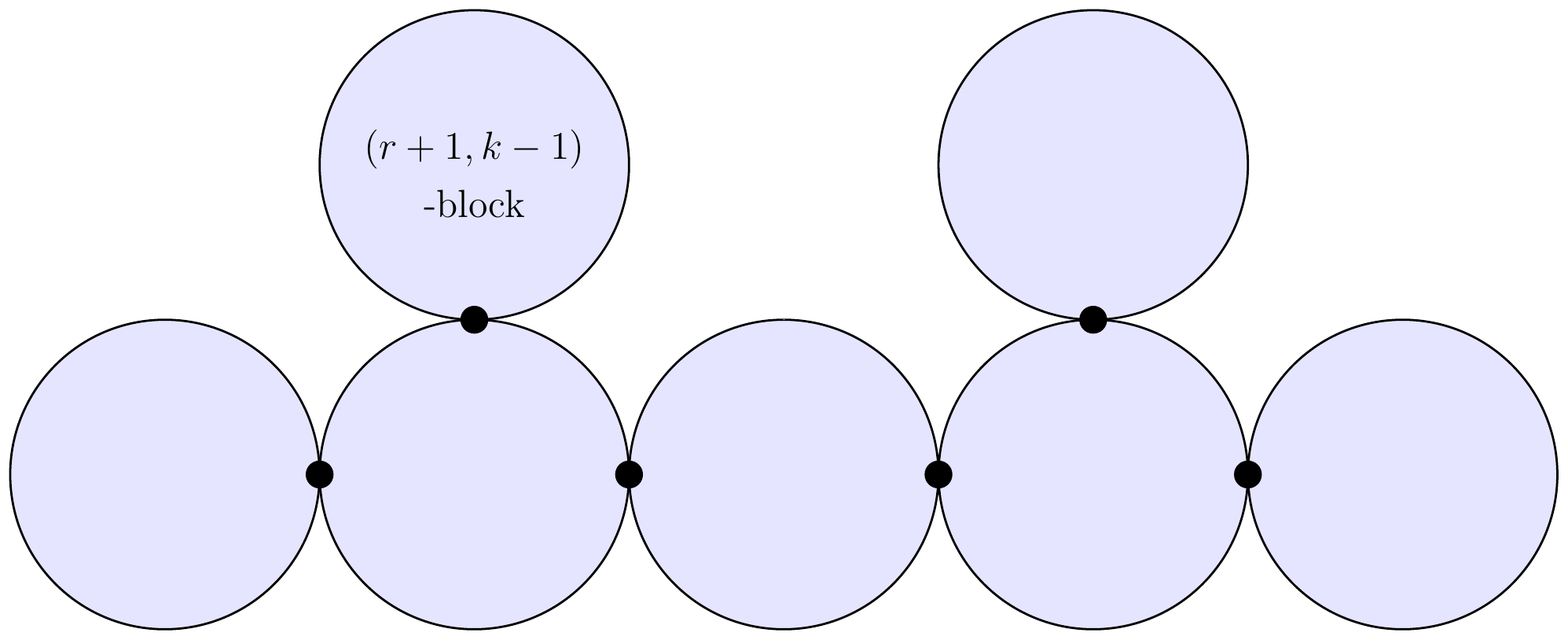}
\caption{An $(r+1,k-1)$-block tree. Each hyperblock contains $r+1$ vertices and $k-1$ hyperedges.}
\end{figure}
An $(r+1,k-1)$-block cannot contain a Berge-cycle of length $k$ or longer because it contains fewer than $k$ edges. Therefore an $(r+1,k-1)$-block-tree also cannot contain such a cycle.

\subsection{ Results for hypergraphs}

%\end{document}
Our main result is:

\begin{thm}\label{EG_full} Let $k \geq 4, r \geq k+1$ and let $\mathcal H$ be  an $n$-vertex  multi-hypergraph  such that $\mathcal H$ has lower rank at least $r$, and each edge of $\mathcal H$ has multiplicity at most $k-2$. 
If  $\mathcal H$ has no Berge-cycles of length $k$ or longer, then $e(\mathcal H)\leq \frac{(k-1)(n-1)}{r}$, and equality holds if and only if $\mathcal H$ is an $(r+1,k-1)$-block-tree.
\end{thm}

As a corollary of Theorem~\ref{EG_full} we obtain a slight generalization of Theorem~\ref{EGp}~\cite{lemons} (their result is
for uniform hypergraphs without repeated edges):

\begin{cor}\label{gyori2}Let $r \geq k+1 \geq 3$, and let $\mathcal H$ be  an $n$-vertex multi-hypergraph  such that $\mathcal H$ has lower rank at least $r$, and each edge of $\mathcal H$ has multiplicity at most $k-2$. If  $\mathcal H$ has no Berge-paths of length $k$, then $e(\mathcal H) \leq \frac{(k-1)n}{r+1}$. \end{cor}

Theorem~\ref{EG_full} also implies the following analogue of the Erd\H{o}s--Gallai theorem for cycles in 
  $r$-uniform hypergraphs (without repeated edges).

\begin{thm}[Erd\H{o}s--Gallai for hypergraphs]\label{EG}
Let $k \geq 4, r \geq k+1$, and let $\mathcal H$ be an $n$-vertex $r$-graph with no Berge-cycles of length $k$ or longer. Then $e(\mathcal H) \leq \frac{(k-1)(n-1)}{r}$. Furthermore, equality holds if and only if $\mathcal H$ is an $(r+1,k-1)$-block-tree.  
\end{thm}

\begin{rem}\label{r1} In an earlier version of this paper, we also proved the asymptotically exact result that
{\em every $n$-vertex multi-hypergraph of lower rank $r$ with no Berge-cycles of length $r$ or longer has fewer than $n$ edges} (i.e., the case for $r=k$).
But since the result for simple $r$-graphs (which is the main question) directly follows from Theorem~\ref{EGMSTZ} by Ergemlidze et al~\cite{EGMSTZ}, we have omitted the parts of
proofs showing this.
%Theorem~\ref{EG} and the results of~\cite{FKL} and ~\cite{EGMSTZ} cover all pairs of $r$ and $k$ except when $r = k$. The methods used in this paper can also prove that $r$-uniform hypergraphs (or multi-hypergraphs with multiplicity at most $r-2$) with no Berge cycles of length $r$ or longer have at most $n-1$ edges. This bound however is not sharp and can also be implied by the $k= r+1$ case of~\cite{EGMSTZ}, and hence we leave out the proof.
\end{rem}

%Unlike other pairs of $r$ and $k$, t
 There is a phase transition when $r = k$. Let $\mathcal H$ be an $r$-uniform hypergraph with 
vertex set $\{v_1,\ldots,v_n\}$  and edge set $\{e_1,\ldots,e_{n-r+1}\}$, where $e_i=\{v_i\}\cup\{v_n,v_{n-1},\ldots,v_{n-r+2}\}$.
%$m=n-k+1$ edges, such that all hyperedges contain the same $(k-1)$-subset of vertices along with one additional (unique) vertex.    
 Then the longest Berge-cycle in $\mathcal H$ has length $r-1$, and $m =n- (r-1)$. Thus when  $r$ is fixed and $n$ is large, 
   $e(\mathcal H)=n-r+1 > \frac{(r-1)(n-1)}{r}$.
   But when $n$ is small, $\frac{(r-1)(n-1)}{r}$ is larger.
   
We believe that this construction and the aforementioned $(r+1,k-1)$-block-trees are optimal. 

\begin{conj}
Let $\mathcal H$ be an $r$-uniform $n$-vertex hypergraph with no cycle of length $r$ or longer. Then $e(\mathcal H) \leq \max\{\frac{r-1}{r}(n-1), n-r+1\}$. 
\end{conj}

The key to our proof is
 a stronger version of Theorem~\ref{EG_full} for multi-hypergraphs that are  2-connected.

\begin{thm}\label{EGblock}Let $k \geq 4, r \geq k+1$ and let $\mathcal H$ be  an $n$-vertex multi-hypergraph  such that $\mathcal H$ is  2-connected, has lower rank at least $r$, and each edge of $\mathcal H$ has multiplicity at most $k-2$. If $\mathcal H$ contains no Berge-cycles of length $k$ or longer, then \[e(\mathcal H) \leq \max\{k-1, \frac{k}{2r-k+2}(n-1)\}.\]
\end{thm}
A proof  similar to that of Corollary~\ref{gyori2} (see the last section) yields the following result for paths in connected hypergraphs.

\begin{cor}Let $r \geq k \geq 3$, and let $\mathcal H$ be a connected   $n$-vertex $r$-graph  with no Berge-path of length $k$. Then \[e(\mathcal H) \leq \max\{k-1,  \frac{k}{2r-k+4}n\}.\] 
\end{cor}

\begin{rem} We do not know if the bound for $e(\mathcal H)$ in Theorem~\ref{EGblock} is sharp. 
But the following multi-hypergraph construction shows that when $k$ is much smaller than $r$, our bound is asymptotically (when $r$ tends to infinity)
optimal: Let $k\geq 3$ be odd, $t \in \mathbb N$ and $V(\mathcal H_t) = \{a, b\} \cup V_1 \cup \ldots \cup V_t$ where $|V_i| = r-2$ for each $1\leq i \leq t$, and the $V_i$'s are pairwise disjoint. The edge set of $\mathcal H_t$ consists of $\frac{k-1}{2}$ copies of $V_i \cup \{a, b\}$ for each $1\leq i \leq t$. Then each Berge-cycle in $\mathcal H_t$ intersects at most two $V_i$'s and hence contains at most $k-1$ hyperedges.  We also have
\[e(\mathcal H_t) = \frac{k-1}{2} t = \frac{k-1}{2r-4}(n-2).\]
\end{rem}

\subsection{ Results for bipartite graphs}
By definition, a multi-hypergraph $\mathcal H$ has a cycle of length $k$ if and only if
the incidence graph $G(\mathcal H)$ has a cycle of length $2k$. Also if $\mathcal H$ has lower rank $r$, then the degree of
each vertex in one of the parts of $G(\mathcal H)$ is at least $r$. In view of this we have studied bipartite graphs 
$G=(X,Y;E)$ with circumference at most $2k-2$ in which the degrees of all vertices in $X$ are at least $r$. One of the main results
(implying Theorem~\ref{EG_full}) is:

\begin{thm}\label{EGbgr} Let $k \geq 4, r \geq k+1$ and let $G=(X,Y;E)$ be a bipartite graph with $|X|=m$ and $|Y|=n$  such that 
$d(x)\geq r$ for every $x\in X$. Also suppose $G$ has no blocks isomorphic to $K_{k-1,r}$. If $c(G)<2k$, then
$m \leq  \frac{k-1}{r}(n-1)$. Moreover, if $m =  \frac{k-1}{r}(n-1)$, then every block of $G$ is a subgraph of $K_{k-1,r+1}$ and
every cut vertex is in $Y$. 
\end{thm}

The heart of the proof is the following stronger bound for $2$-connected graphs.

\begin{thm}\label{EGbgr2} Let $k \geq 4, r \geq k+1$ and let $G=(X,Y;E)$ be a bipartite $2$-connected graph with $|X|=m$ and $|Y|=n$  such that $m\geq k$ and
$d(x)\geq r$ for every $x\in X$.  If $c(G)<2k$, then
$m \leq  \frac{k}{2r-k+2}(n-1)$. 
\end{thm}

%It turned out that in order to have a stronger induction assumption, we need to prove a somewhat more general statement.

In order to use induction on the number of blocks, we will prove a more general statement:
We will allow some vertices in $X$ to have degrees less than $r$ and assign them a {\em deficiency}.

% when some vertices in $X$
%are allowed to have degree less than $r$.
% Theorem~\ref{EGbgr}, we would like to apply induction to the number of blocks and use Theorem~\ref{EGbgr2}. It turns out that a problem occurs when there exists a cut vertex $x$ in $X$ such that when we split $G$ into multiple blocks, the degree of $x$ drops below $r$. In this case, we allow some vertices in $X$ to have degrees less than $r$ and assign them a {\em deficiency}. For this, we need the following definitions.

%\begin{definition}
 Let $G = (X, Y; E)$ be a bipartite graph, and $r$ be a positive integer. For a vertex $x \in X$,  the {\em deficiency} of $x$ is $D_G(x) := \max\{0, r - d_G(x)\}$. For a subset $X^* \subseteq X$, define the 
 {\em deficiency of $X^*$} as $D(G, X^*) := \sum_{x \in X^*} D_G(x)$. 
%\end{definition}

In these terms our more general theorem is as follows.

%Our goal will be eventually to prove the Main Theorem:
\begin{thm}[Main Theorem]\label{main} 
Let $ k\geq 4$, $r\geq k+1$  and $m,m^*,n$ be positive integers with $n\geq k$, $m \geq m^*\geq  k-1$ and $m\geq k$. 
Let $G=(X,Y;E)$ be a bipartite $2$-connected  graph with parts $X$ and $Y$, where $|X|=m$, $|Y|=n$, and
let $X^*\subseteq X$ with $|X^*|=m^*$.  If $c(G)<2k$, then 
\begin{equation}\label{e2}
 m^*\leq \frac{k}{2r-k+2}(n-1+D(G, X^*)).
\end{equation}
\end{thm}

\section{Proof outline}
%We consider the incidence graph $G = G(\mathcal H)$ of a multi-hypergraph satisfying the hypothesis of Theorem~\ref{EG_full}. 
%For a multi-hypergraph $\mathcal H$ satisfying the hypothesis of Theorem~\ref{EG_full}, we construct an auxiliary bipartite graph $G = G(\mathcal H)$ as follows: $G = X \cup Y$ where $X = E(\mathcal H)$ and $Y = V(\mathcal H)$, and $xy \in e(G)$ if and only if the edge $x$ in $\mathcal H$ contains the vertex $y$.
%By definition, $\mathcal H$ contains a Berge-cycle of length $\ell$ if and only if $G$ contains a cycle of length $2\ell$. Thus we obtain that $G$ contains no cycle of length $2k$ or longer. By construction, since $\mathcal H$ has lower rank at least $r$, each $x \in X$ has $d_G(x) \geq r$.

%Instead of proving results for hypergraphs directly, we  work with bipartite graphs with bounded circumference. Moreover, the bulk of the proof will be theorems for 2-connected bipartite graphs with no long cycles, which we will then use to prove weaker bounds for general bipartite graphs using induction.

As we discussed in the previous section, our main theorem is on bipartite graphs with circumference at most $2k-2$, based on a stronger result for $2$-connected graphs.

In Section 4, we present a general  theorem on the structure of 2-connected bipartite graphs with no long cycles and the most edges. In particular, we show that for $3 \leq d \leq (k-1)/2$, each such graph that is neither $d$-degenerate nor ``too dense" contains a substructure that  we call a ``saturated crossing formation". In Section 5, we state the Main Theorem for 2-connected bipartite graphs that will be used to prove the inductive statement for general bipartite graphs. In Sections 5, 6, and 7 we show that if a graph is too sparse then it satisfies the Main Theorem, but if it is too dense, then it  contains a long cycle. So our graphs must contain a path in saturated crossing formation, but we also prove that any graph that contains such a path  satisfies our Main Theorem,  a contradiction. In Section 8, we prove a bound on the size of $X$ for general bipartite graphs, and in Section 9, we apply this bound to finally prove Theorem~\ref{EG_full} for hypergraphs.

\section{Structure of bipartite graphs without long cycles}

\begin{definition}Let $P = v_1, \ldots, v_p$ be a path with endpoints $x= v_1$ and $y = v_p$. For vertices $v_i, v_j$ of $P$, let $P[v_i, v_j] = v_i,v_{i+1}, \ldots, v_j$ if $i < j$, and $v_i, v_{i-1}, \ldots, v_j$ if $i > j$. 
\end{definition}

\begin{definition}Vertices $v_i$, $v_j$ in $P$ are called {\bf crossing neighbors} if $i < j$, $v_i \in N(y), v_j \in N(x)$, and for each $i < \ell < j$, $v_\ell \notin N(x) \cup N(y)$. The  edges $xv_j, yv_i$ are called {\bf crossing edges}.
\end{definition}

\begin{definition}For a set $S \subseteq V(P)$, define $S^+_P := \{v_{i+1}: v_i \in S\}$ and $S^-_P := \{v_{i-1}: v_i \in S\}$. When there is no ambiguity, we will simply write $S^+ = S^+_P$ and $S^- = S^-_P$. 
\end{definition}

%
%\begin{lem}\label{cycle}
%Let $G$ be a 2-connected bipartite graph, and let $P$ be an $(x,y)$-path. Then either
%\begin{enumerate}
%\item[(a)] $x$ and $y$ are in different partite sets of $G$ and there exists a cycle of length at least \[\min\{|V(P)|, 2(d_P(x) + d_P(y)-1)\},\] or
%\item[(b)] $x$ and $y$ are in the same partite set and there exists a cycle of length at least \[\min\{|V(P)| - 1, 2(d_P(x) + d_P(y) - 2)\}.\] 
%\end{enumerate}
%\end{lem}

\begin{lem}\label{cycle}
Let $G$ be a 2-connected bipartite graph, and let $P$ be an $(x,y)$-path. Then either
\begin{enumerate}
\item[(a)] $P$ contains no crossing neighbors and $G$ has a cycle of length at least $2(d_P(x) + d_P(y)-1)$, or
\item[(b)] $x$ and $y$ are in different partite sets of $G$ and there exists a cycle of length at least \[\min\{|V(P)|, 2(d_P(x) + d_P(y)-1)\}\] in $G$, or
\item[(c)] $x$ and $y$ are in the same partite set and there exists a cycle of length at least \[\min\{|V(P)| - 1, 2(d_P(x) + d_P(y) - 2)\}\] 
in $G$. 
Furthermore in all cases, we obtain a cycle that covers $N_P(x) \cup N_P(y)$. 
\end{enumerate}
\end{lem}

\begin{proof}
Suppose first that $P$ contains no crossing neighbors. Our proof is based off Bondy's theorem for general 2-connected graphs.

Let $P = v_1, \ldots, v_p$ where $v_1 = x, v_p = y$. Let $t_0 = \max\{s: v_s \in N(x)\}$ and $u = \min\{s:v_s \in N(y)\}$, thus $t_0 \leq u$. Iteratively construct paths $P_1, P_2, \ldots$ as follows: given $t_{r-1}$, find $s_r, t_r$ such that $s_r < t_{r-1} < t_r$ where $t_r$ is as large as possible, and $P_r$ is a path from $v_{s_r}$ to $v_{t_r}$ that is internally disjoint from $P$. It is always possible to find such a $P_r$ because $G$ is 2-connected. We stop at step $\ell$ at the first instance where $\ell > u$. Observe that for $r_1 < r_2$, paths $P_{r_1}$ and $P_{r_2}$ must be disjoint: if they share a vertex, then we would have chosen $P_{r_1}$ to end at vertex $v_{r_2}$, contradicting the maximality of $r_1$. Also, $s_{r+1} \geq t_{r-1}$, otherwise we would choose $P_{r+1}$ instead of $P_r$.

Now let $a= \min\{r: v_r \in N(x), r > s_1\}$, $b = \max\{r: v_r \in N(y), r < t_\ell\}$. 

If $\ell$ is odd, then we take the cycle
\[C_1:=P[x, v{s_1}] \cup P_1 \cup P[v_{t_1}, v_{s_3}] \cup P_3 \cup \ldots \cup P[v_{t_{\ell-2}}, v_{s_\ell}] \cup P_\ell \cup P[v{t_\ell}, y] \cup \]
\[ \cup yv_u \cup P[v_u, v_{t_{\ell - 1}}] \cup P_{\ell - 1} \cup \ldots \cup P[v_{s_4}, v_{t_2}] \cup P_2 \cup P[v_{s_2}, v_{t_0}] \cup v_{t_0}x.\]
And if $\ell$ is even, we take the cycle
\[C_2:=P[x, v_{s_1}] \cup P_1 \cup P[v_{t_1}, v_{s_3}] \cup P_3 \cup \ldots \cup P[v_{t_{\ell-3}}, v_{s_{\ell-1}}] \cup P_{\ell-1} \cup P[v_{t_{\ell-1}}, v_u] \cup \]
\[ \cup v_uy \cup P[y, v_{t_{\ell}}] \cup P_{\ell} \cup \ldots \cup P[v_{s_4}, v_{t_2}] \cup P_2 \cup P[v_{s_2},v_{t_0}] \cup v_{t_0}x.\]

\begin{center}
\includegraphics[scale=.5]{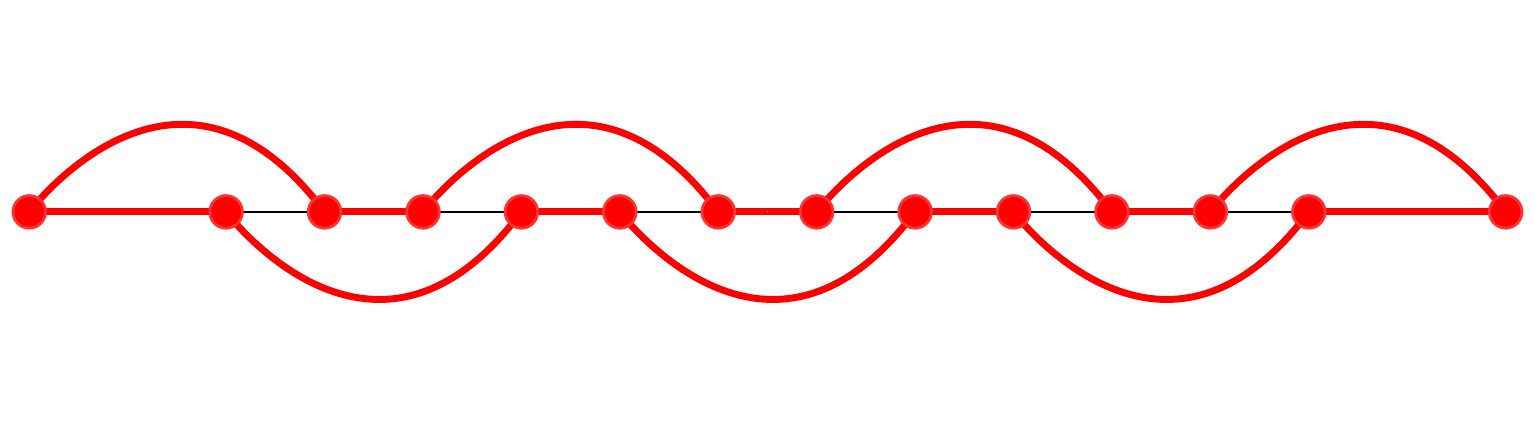}
\end{center}

Both cycles cover $N_P(x) \cup N_P(y)$. If $x$ and $y$ are the same parity, since $P$ contains no crossing neighbors, $|N_P(x) \cap N_P(y)|\leq 1$. Therefore $P$ contains at least $d_P(x) + d_P(y) - 1$ even vertices, which implies $|V(C)| \geq 2(d_P(x) + d_P(y) -1)$ because $G$ is bipartite. Otherwise, if $x$ and $y$ are different parities, then the neighbors of $x$ and the successors of  neighbors of $y$ are disjoint, of the same parity, and are contained in $C$. Thus $|V(C)| \geq 2(d_P(x) + d_P(y))$, as desired. 

Now suppose $G$ has crossing neighbors $v_i$ and $v_j$ with $j < i$ and $xv_j, yv_i \in E(G)$.  Let $C = P[x, v_i] \cup v_i y \cup P[y, v_j] \cup v_j x$. If $j =i+1$, then $C$ contains all vertices of $P$, as desired. If $j = i+2$, then $x$ and $y$ must be the same parity, and $C$ omits only vertex $v_{i+1}$. I.e., $|V(C)| = |V(P)|-1$.

Consider first the case where $x$ and $y$ are different parities and each pair of crossing neighbors has at least 2 vertices between them. For every neighbor $v_s$ of $x$ in $P - \{v_j\}$ (note $s$ is even), the odd vertex $v_{s-1}$ is in $C$ and is not a neighbor of $y$, otherwise $v_{s-1}$ and $v_s$ would form a pair of crossing neighbors.  Also, each neighbor of $y$ in $P$ is in $C$. Thus $C$ has at least $d_P(x) - 1 + d_P(y)$ odd vertices. That is, $|C| \geq 2(d_P(x) + d_P(y) - 1)$.

Now suppose $x$ and $y$ are the same parity and that crossing neighbors have at least 3 vertices between them. Let $C$ be as before.

%For each vertex $v_t \in N_P(x)$, $v_t \in V(C)$. For each vertex $v_t \in N_P(y) - \{v_i, v_{p-1}\}, v_{t+2} \in V(C)$ and $v_{t+2}$ is not in $N_P(x)$, since $v_t, v_{t+2}$ would form crossing neighbors with only one vertex between them. Because each vertex in $N_P(x) \cup N_P(y)$ must be even parity, we have at least $(k+1)/2 + (k+1)/2 - 2 = k-1$ even vertices in $C$. Therefore $|V(C)| \geq 2(d_P(x) + d_P(y) - 2)$ vertices, as desired. 

For any vertices $v_s \in N_P(x) - \{v_j\}$ and $v_t \in N_P(y) - \{v_i\}$, $v_{s-1}$ and $v_{t+1}$ are distinct and of the same parity (in this case, odd). Thus $C$ contains at least $d_P(x) -1 + d_P(y)-1$ odd vertices. It follows that $|C| \geq 2(d_P(x) + d_P(y) - 2)$, as desired.
\end{proof}

\begin{definition} Let $G$ be a 2-connected bipartite graph, and let $P$ be a path $v_1,\ldots,v_p$.
We say that  $P$ is in {\bf crossing formation} if there is a sequence of vertices $v_{i_0}, v_{i_1}, \ldots, v_{i_q}$ such that $v_i, v_{i'}$ are crossing neighbors if and only if $\{v_i, v_{i'}\} = \{v_{i_\ell}, v_{i_{\ell + 1}}\}$ for some $0 \leq \ell \leq q-1$. \end{definition}

\begin{definition} Let $G$ be a 2-connected bipartite graph, and let $P$ be an $x,y$-path $v_1,\ldots,v_p$.
Say that $P$ is in {\bf  saturated crossing formation} if 
\begin{enumerate}
\item $P$ is in crossing formation,
%\item for each even $i\leq i_1$, vertex   $v_{i}$ belongs to $N(v_1) \setminus N(v_p)$ and for each even $h\geq i_j$, vertex   $v_{h}$ belongs to $N(v_p) \setminus N(v_1)$,
\item $G[\{v_1, v_2, \ldots, v_{i_0}\} \cup \{v_{i_0}, \ldots, v_{i_q}\}]$ and $G[\{v_{i_q}, \ldots, v_p\}  \cup \{v_{i_0}, \ldots, v_{i_q}\}]$ are both complete bipartite,
\item if $P$ has more than one pair of crossing neighbors, then each pair has exactly 3 vertices between them,
\item $v_2$ has a neighbor $v_1'$ outside of $P$ such that $N(v'_1)=N(v_1)$, and
$v_{p-1}$ has a neighbor $v_p'$ outside of $P$ such that $N_P(v'_p)=N_P(v_p)$,
\item  for every even $h\leq i_0-2$ and every $u\in N_G(v_h)$, $N_G(u)\subseteq N_P(v_1)$,
similarly,
for every even $h\geq i_q+2$ and every $w\in N_G(v_h)$, $N_G(w)\subseteq N_P(v_p)$; 
in particular, for every odd $g\leq i_0-1$,  $N_G(v_g)\subseteq  V(P)$
and
for every odd $h\geq i_q+1$,  $N_G(v_h)\subseteq  V(P)$.

\end{enumerate}
\end{definition}

\begin{figure}[!ht]
\centering
\includegraphics[scale=.4]{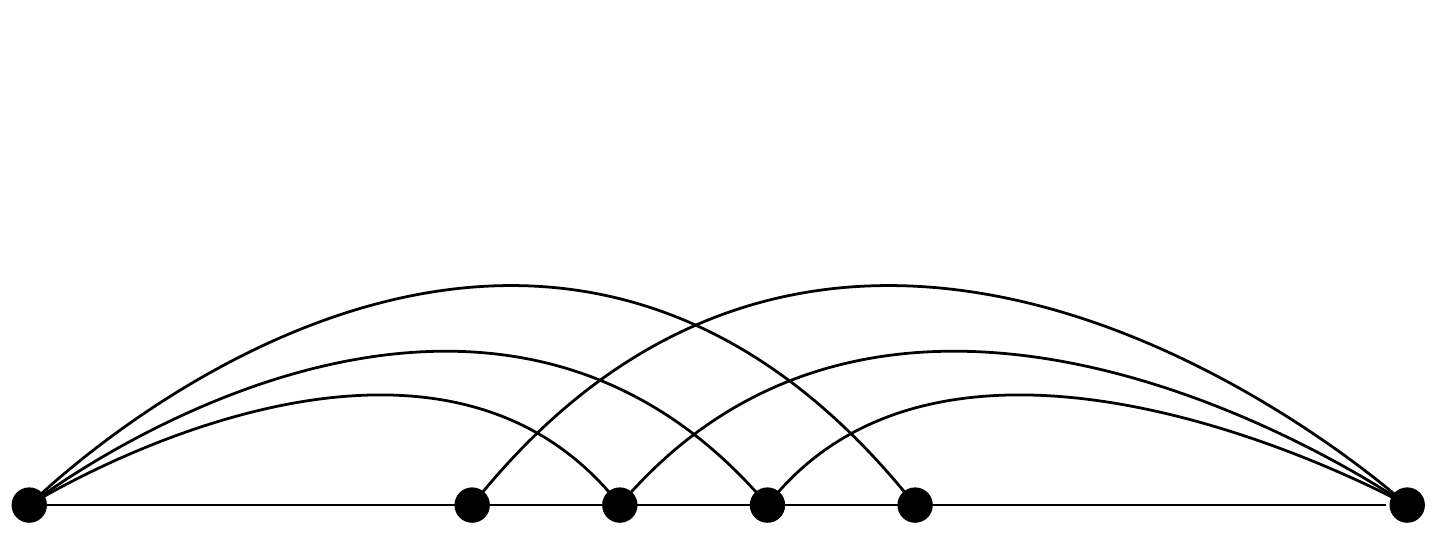}
\includegraphics[scale=.4]{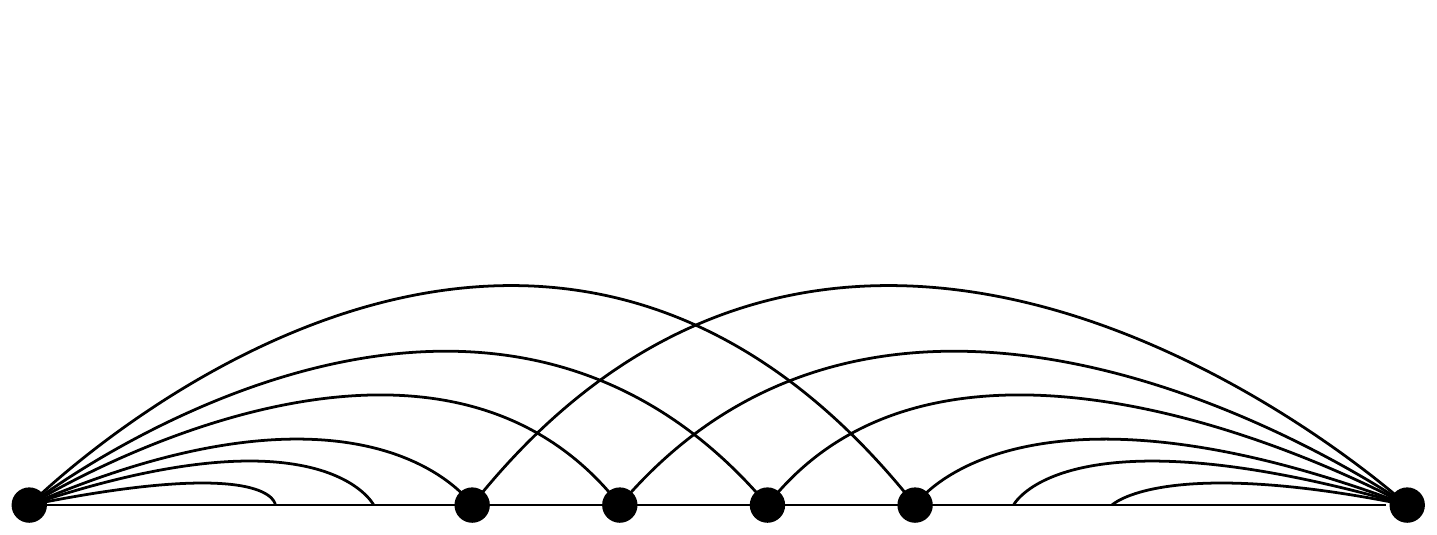}
\caption{A crossing formation, a saturated crossing formation}
\end{figure}

\begin{definition}
For a bipartite graph $G = A \cup B$, $\alpha \in \mathbb N$, and subsets $X^* \subseteq A$, $Y^* \subseteq B$, the {\bf $\alpha(X^*, Y^*)$-disintegration} of $G$ is the process of first deleting the vertices of $(A - X^*) \cup (B - Y^*)$ from $G$, then iteratively removing the remaining vertices of degree at most $\alpha$ until the resulting graph is either empty or has minimum degree at least $\alpha+1$. Let $G_\alpha(X^*, Y^*)$ denote the result of applying $\alpha(X^*, Y^*)$-disintegration to $G$. 

In the case where $X^* = A$ and $Y^* = B$, in literature, $G_\alpha(X^*, Y^*)$ is commonly referred to as the $(\alpha+1)$-core of $G$, that is, the unique maximum subgraph of $G$ with minimum degree at least $\alpha + 1$.
\end{definition}
Note that if $\alpha \geq (k-1)/2$, then $k-1 - \alpha \leq \alpha$, so $G_\alpha(X^*, Y^*) \subseteq G_{k-1-\alpha}(X^*, Y^*)$. 

\begin{definition}A bigraph $G=(X,Y;E)$
 is
$2k$-{\bf saturated} if $c(G)<2k$, but for each $x\in X$ and $y\in Y$ with $xy\notin E(G)$, the graph $G+xy$ has a cycle of length at least $2k$.\end{definition}

For example, if $s\leq k-1$ then for any $t$, the complete bipartite graph $K_{s,t}$ is $2k$-saturated, because it does not have
$x\in X$ and $y\in Y$ with $xy\notin E(G)$

\begin{thm}\label{cycle1}
Fix $k \geq 3$ odd and let $G$ be a 2-connected $2k$-saturated bipartite graph. For some $(k-1)/2 \leq \alpha \leq k-2$, and $(X^*, Y^*)$, suppose there are $x\in V(G_\alpha(X^*, Y^*))$ and $y\in V(G_{k-1-\alpha}(X^*, Y^*))$ that are nonadjacent to each other. Let $P=v_1,\ldots,v_p$ be a path with the following properties: (1) $P$ is a longest path in $G$ such that $v_1 \in G_\alpha(X^*, Y^*)$ and $v_p \in G_{k-1-\alpha}(X^*, Y^*)$, and (2) subject to the first condition, $\sum_{i=1}^p d_P(v_i)$ is maximized. Then $P$ is in a saturated crossing formation.
\end{thm}

\begin{proof}For simplicity, denote $G_\alpha = G_\alpha(X^*, Y^*)$ and $G_{k-1-\alpha} = G_{k-1-\alpha}(X^*, Y^*)$. 

Because $G$ is saturated and $xy\notin E(G)$, $G$ contains an $(x,y)$-path with at least $2k$ vertices. Thus $p\geq 2k$.
 By the maximality of $P$, all neighbors of $v_1$  in $G_\alpha$ and all neighbors of $v_p$  in $G_{k-1-\alpha}$
  are in $P$. Thus 
 \begin{equation}\label{m2}
\mbox{
  $d_P(v_1) \geq \alpha+1$ and $d_P(v_p) \geq k-1- \alpha+1$. }
\end{equation}
  
  By Lemma~\ref{cycle}, $G$ contains a cycle of length at least $2(d_P(x) + d_P(y) - 2) \geq 2((\alpha + 1) + (k - 1 - \alpha + 1) - 2) = 2(k-1)$. But $c(G) \leq 2k-2$, so $P$ satisfies neither (a) nor (b) in Lemma~\ref{cycle}. In particular, $p$ is odd (so $p \geq 2k+1$), and $G$ has crossing neighbors. Let $v_i,v_j$ be a pair of crossing neighbors such that $v_i - v_j$ is minimized. Examining the proof of Lemma~\ref{cycle}, each pair of crossing neighbors in $P$ has at least 3 vertices between them. Furthermore, we obtain a cycle $C = P[v_1, v_i] \cup v_i v_p \cup P[v_p, v_j] \cup v_j v_1$ such that 
\begin{itemize}
\item[I.] $V(C) =  V(P) - \{v_{i+1}, \ldots, v_{j-1}\}$,
\item[II.] $|V(C)| =  2(d_P(x) + d_P(y) - 2)$, and
\item[III.] each odd vertex in $C$ belongs to $N_P(v_1)^- \cup N_P(v_p)^+$ with $N_P(v_1)^- \cap N_P(v_p)^+ = \emptyset$.
\end{itemize}
In particular, since $C$ misses only  vertices between one pair of crossing neighbors, if $C$ contains more than one pair of crossing neighbors, then each pair only contains 3 vertices between them, otherwise condition III. is violated.  Thus Part 3 in the definition of saturated crossing formation holds.

%For a set $S\subseteq V(P)$, let $S_P^+$ (respectively, $S_P^-$) denote the set of successors (respectively, predecessors) of the vertices in $S$.
%By the construction of $C$, $(N_P(v_1))^-_P-v_{j-1}\subset V(C)$ and $(N_P(v_p))^+_P-v_{i+1}\subset V(C)$.
%If some $v_h$ is in $(N_P(v_1))^-_P\cap (N_P(v_p))^+_P$, then the cycle
%$v_1v_{h+1}\cup P[v_{h+1},v_p]\cup v_pv_{h-1}\cup P[v_1,v_{h-1}]$ is longer than $C$, a contradiction. Thus
%\begin{equation}\label{m5}
%(N_P(v_1))^-_P\cap (N_P(v_p))^+_P=\emptyset,
%\end{equation}
%and hence~\eqref{m2} implies that
% \begin{equation}\label{m3}
%\mbox{each odd vertex in $C$ is in $(N_P(v_1))^-_P\cup (N_P(v_p))^+_P  $.}
%%d_P(v_1) \geq \alpha+1$ and $d_P(v_p) \geq k-1- \alpha+1$. }
%\end{equation}
%The predecessors of all neighbors of $v_1$ in $P$ apart from $v_{j-1}$ are in $C$. Similarly, $C$ contains at least $k-1-\alpha$
%successors of the neighbors of $v_p$ in $P$. Since no
%%Cycle $C$ has $k-1$ odd vertices. Exactly 
%Since $v_2v_p\notin E(G)$ (because otherwise $P[v_2, v_p] \cup v_2v_p$ is a cycle with at least $2k$ vertices), $v_3\notin N_P(v_p)^+ $, and hence 
%$v_3\in N_P(v_1)^- $. This means
%\begin{equation}\label{m6}
%\mbox{
%$v_1v_4\in E(G)$.  Similarly, $v_pv_{p-3}\in E(G)$. }
%\end{equation}

First we show that 
\begin{equation}\label{m1}
\mbox{
$v_2$ has a neighbor $v_1' \in G_\alpha$ outside of $P$. }
\end{equation}

By Lemma~\ref{cycle}, $v_1$ has exactly $\alpha+1$ neighbors in $P$. So by the maximality of $P$, each  of these neighbors must be in $G_\alpha$. 
In particular, $v_2 \in V(G_\alpha)$, and so it has at least $\alpha + 1$ neighbors in $G_\alpha$ as well. Suppose  that all of its neighbors 
in $G_{\alpha}$ are in $P$.

If $v_2$ has a neighbor $v_t \in N(v_p)^+$, then $P[v_2,v_{t-1}] \cup v_{t-1}v_p \cup P[v_p, v_t] \cup v_t v_2$ is a cycle of length at $|V(P)|-1 \geq 2k$, a contradiction. So $N(v_2)\cap N(v_p)^+=\emptyset$. % in $H$ can be a successor of a neighbor of $v_p$. 
Hence by fact III, $N_{G_{\alpha}}(v_2)\subset N_P(v_1)^-\cup V(P-C)$. Since
$v_{j-1}\in N_{G_{\alpha}}(v_1)^-$ but $v_{j-1} \notin V(C)$, we have 
$|N_P(v_1)^- \setminus V(P-C)|\leq \alpha$. Thus $v_2$ has a neighbor $v_h\in V(P-C)$. Furthermore, because $v_h \notin N(v_p)^+$, $i+3\leq h\leq j-1$.

%We claim that $x$ and $y$ share at least 2 neighbors. First, let $v_\ell$ be the first neighbor of $y$ that appears in $P$. If $v_{\ell} \notin N(x)$, then by choice of $v_{\ell}$, $v_{\ell-2} \notin N(y)$, violating III. Similarly, if $v_\ell$ is the last neighbor of $v_1$ in $P$, then $v_{\ell + 2} \notin N(x)$, so $v_{\ell} \in N(y)$. 

%Thus $x$ has at most $(\alpha +1) - 2$ neighbors that are not neighbors of $y$. So because $|N_P(v_2)| \geq \alpha+1$, $v_2$ has at least one neighbor $v_{2'}$ that is neither $x$ nor a successor of a neighbor of $x$. Then $v_{2'}$ must lie between a pair of crossing neighbors $v_s, v_t$ (not necessarily $v_i$,$v_j$ but $i \leq s, j \leq t$). 
Let $v_\ell$ be the first neighbor of $v_p$ that appears in $P$ (so $v_{\ell - 2} \in N(v_1)$ by III). 
Then the cycle
\[C' := v_2v_{h} \cup P[v_{h}, v_\ell] \cup v_\ell v_p \cup P[v_p, v_j] \cup v_j v_1 \cup v_1 v_{\ell - 2} \cup P[v_{\ell - 2}, v_2]\]
has at least $|V(C)| + 2$ vertices, a contradiction. This proves~\eqref{m1}.

%Therefore $v_2$ has a neighbor $x'$ in $G_\alpha$ outside of $P$. 
\smallskip
Consider the path $P' = v'_1v_2 \cup P[v_2, v_p]$.  By definition, $|V(P')| = |V(P)|$. By the maximality of $P$,
% as a longest path with nonadjacent endpoints in $H$,
 each neighbor of $v'_1$ in $G_\alpha$ must also lie in $P'$, and $v'_1$ must have $\alpha + 1$ such neighbors, otherwise we would apply Lemma~\ref{cycle} to get a longer cycle in $G$. 

Suppose $v'_1$ is adjacent to $v_h$ for some $i+1\leq h\leq j-1$. 
% By~\eqref{m6}, $v_1v_4\in E(G)$.
Then the cycle $P'[v_1', v_i] \cup v_iv_p \cup P'[v_p, v_h] \cup v_h v_1'$ is longer than $C$, a contradiction.
Thus $N(v'_1)\cap P\subseteq V(C)$. Then the analog of fact III for $P'$ yields
 \begin{equation}\label{m7}
N_P(v'_1)  =N_P(v_1).
%d_P(v_1) \geq \alpha+1$ and $d_P(v_p) \geq k-1- \alpha+1$. }
\end{equation}
By a symmetric argument, we get the analog of~\eqref{m1} and~\eqref{m7}:
\begin{equation}\label{m8}
\mbox{
$v_{p-1}$ has a neighbor $v_{p}'$ outside of $P$ such that    $N_P(v'_p)  =N_P(v_{p})$. }
\end{equation}
This shows that Part 4 of the definition of saturated crossing formation holds.

%First suppose $v_j \notin N(x')$. Then because $P'$ must satisfy condition (c) of Lemma\ref{cycle}, we can find another pair of cross neighbors $v_{i'}, v_{j'}$ and a cycle $C'$ such that $V(C') = V(P) - \{v_{i'+1}, \ldots, v_{j'-1}\}$ and $|V(C')| = 2(k-1) = |V(C)|$. In particular, $v_j \in V(C')$. But $v_j \notin N(x')$ and $v_{j-2} \notin N(y)$, violating condition III. (applied to $P'$). 

%Therefore $v_i$ and $v_j$ are also a pair of crossing neighbors in $P'$, we may construct $C'$ such that it only avoids vertices between this pair, and I., II., and III. apply to $C'$. To satisfy condition III., all even vertices $v_{t+2}$ such that $v_t \notin N(y) - \{v_i, v_{p-1}\}$ must be neighbors of $x'$. But these vertices are also neighbors of $x$. So we obtain that $N_P(x) = N_{P'}(x') = N_P(x')$. Similarly, $v_{p-1}$ has a neighbor $y' \in G_{k-1-\alpha}$ outside of $P$ such that $N_P(y) = N_P(y')$. 

Again, let $v_\ell$ be the first neighbor of $v_p$ in $P$. We claim that
\begin{equation}\label{m9}
\mbox{for each even $h\geq \ell$, either $v_1v_h\notin E(G)$ or $v_1v_{h+2}\notin E(G)$.}
\end{equation}
Indeed, suppose $h\geq \ell$, $v_1v_h\in E(G)$ and $v_1v_{h+2}\in E(G)$. Then by~\eqref{m7}, $v'_1v_h\in E(G)$ and by the definition of $\ell$,
$v_1v_{\ell-2}\in E(G)$.  Then the cycle
\[C'' := v'_1v_{h} \cup P[v_{h}, v_{\ell}] \cup v_{\ell}v_p \cup P[v_{p}, v_{h+2}] \cup v_{h+2} v_1 \cup v_1 v_{\ell-2} \cup P[v_{\ell- 2}, v_{2}] \cup v_{2}v_1'\]
avoids only the vertices $v_{\ell - 1}$ and $v_{h + 1}$ in $P$ and includes  $v_1'\notin P$. Thus $|V(C'')| \geq  2k$, a contradiction.

Similarly , if $v_{\ell'}$ is the last neighbor of $v_1$ in $P$, then
\begin{equation}\label{m10}
\mbox{for each even $h\leq \ell'$, either $v_pv_h\notin E(G)$ or $v_pv_{h-2}\notin E(G)$.}
\end{equation}
%Thus $y$ cannot have any neighbors in $P$ before $v_i$. Similarly, one can show that $x$ cannot have any neighbors after $v_j$ that are not also contained in a pair of crossing neighbors. 
Together,~\eqref{m9} and~\eqref{m10} imply   Part 1 of the definition of the saturated crossing formation holds, i.e.
there is a sequence of vertices $v_{i_0}, v_{i_1},, \ldots, v_{i_q}$ with $i_0=i$ and $i_1=j$ such that $v_r, v_{r'}$ are crossing neighbors if and only if $\{v_r, v_{r'}\} = \{v_{i_\ell}, v_{i_{\ell + 1}}\}$ for some $0 \leq \ell \leq q-1$. To see this, suppose there exists two pairs of crossing neighbors, $\{v_{a_1}, v_{b_1}\}$ and $\{v_{a_2}, v_{b_2}\}$ such that there are no other pairs of crossing neighbors between them, and $b_1 < a_2$. Then $a_2 \geq b_1 + 2$. Then by \eqref{m9}, $v_{b_1 + 2}v_1 \notin E(G)$. If $v_{b_1}v_p \notin E(G)$, then the vertex $v_{b_1+1}$ violates fact III. Otherwise, if $v_{b_1}v_p \in E(G)$, then because there are no crossing pairs between $\{v_{a_1}, v_{b_2}\}$ and $\{v_{a_2}, v_{b_2}\}$, for each $b_1< c \leq a_2$, $v_cv_1 \notin E(G)$. By condition III, this means each even vertex $v_c$ between $v_{b_1}$ and $v_{a_2}$ belong to $N(v_p)$, contradicting \eqref{m10}. 

Therefore we have proved that $P$ is in crossing formation (Part 1 of the definition of saturated crossing formation). Let $v_{i_0}, \ldots, v_{i_q}$ be the set of crossing neighbors. By fact III,

\begin{equation}\label{m10'}
\mbox{for each even $s \leq i_0$, $v_1v_s \in E(G)$, and for each even $t \geq i_q$, $v_pv_t \in E(G)$. }
\end{equation}

Next we will prove Part 5 in $3$ steps. Our first step is to prove:
\begin{equation}\label{m12}
\mbox{For each odd $1\leq h<i_0$, $N(v_h)\subseteq N_P(v_1)$. In particular, $N(v_1)=N_P(v_1)$.}
\end{equation}
Indeed, suppose for some odd $1\leq h<i_0$, vertex $v_h$ has a neighbor $w\notin \{v_2,v_4,\ldots,v_{i_0-2}\}\cup \{v_{i_0},\ldots,v_{i_q}\}$.
Since $G$ is $2$-connected, $G-v_h$ contains a path $Q=w_1,\ldots,w_s$ from $w=w_1$ to $P-v_h+v'_1$ (possibly, $s=1$ if $w \in P$) that is internally disjoint from $P$. 
If $w_s=v'_1$, then the path $$P_1=Q^{-1}\cup wv_h\cup P[v_1,v_h]\cup v_1v_{h+1}\cup P[v_{h+1},v_p]$$
starts from $v'_1\in G_\alpha$ and is longer than $P$, a contradiction.
Suppose now that
 $w_s=v_g$. Then $v'_1\notin V(Q)$.
If $g>i_q$, then the cycle $$C_1=v_hw\cup Q\cup P[v_g,v_p]\cup v_pv_{2\lfloor (g-1)/2\rfloor}\cup P[v_{2\lfloor (g-1)/2\rfloor},v_{h+1}]
\cup v_{h+1}v_1\cup P[v_1,v_h]$$ has at least $2k$ vertices, a contradiction. If $i_j<g\leq i_{j+1}$ for some $1\leq j<q$, then the cycle 
$$C_2=v_hw\cup Q\cup P[v_g,v_p]\cup v_pv_{i_j}\cup P[v_{h+1},v_{i_j}]
\cup v_{h+1}v_1\cup P[v_1,v_h]$$ 
is longer than $C$, unless $g=i_{j+1}$ and $s=1$. But $g=i_{j+1}$ and $s=1$ means $w=v_{i_{j+1}}$, contradicting the fact that $w\notin N_P(v_1)$.
So suppose $1\leq g\leq i_0$. If $|g-h|=1$ then $ s\geq 2$: if $s = 1$ then $Q = w = i_g \in N_P(v_1)$, a contradiction. But if $s \geq 2$, then replacing edge $v_hv_g$ in $P$ with $Q$, we obtain a longer $(v_1,v_p)$-path. So let $|g-h|\geq 2$.
Since $v'_1\notin V(Q)$, if $g>h$, then the path
$$P_1=P[v_1,v_{h-1}]\cup v_{h-1},v'_1\cup v'_1v_{2\lfloor (g-1)/2\rfloor}\cup P[v_{2\lfloor (g-1)/2\rfloor},v_h]\cup v_hw
\cup Q\cup P[v_g,v_p]$$
 has the same ends as $P$, but is longer than $P$, contradicting the choice of $P$. Similarly, if $g<h$, then the path
$$P_2=P[v_1,v_{2\lfloor (g-1)/2\rfloor}]\cup v_{2\lfloor (g-1)/2\rfloor},v'_1\cup v'_1v_{h-1}\cup P[v_g,v_{h-1}]\cup Q
\cup wv_h\cup P[v_h,v_p]$$
 has the same ends as $P$, but is longer than $P$. This proves~\eqref{m12}.
 
By the symmetry between $v_1$ and $v'_1$, the same proof implies
\begin{equation}\label{m12'}
 N(v'_1)=N_P(v_1).
\end{equation}

Now let $h<i_0$ be even.  %The first step towards proving Part 5 is the following claim.
\begin{equation}\label{m11}
\mbox{For any $g\in \{i_0,\ldots,p\}\setminus \{i_0,\ldots,i_q\}$, there is no $v_h,v_g$-path internally disjoint from $P$.}
\end{equation}
Indeed, suppose such a path $Q=w_1,w_2,\ldots,w_s$ exists with $w_1=v_h$ and $w_s=v_g$.
If $i_{j}<g<i_{j+1}$ for some $1\leq j<q$, then the cycle
$$C_3=Q\cup P[v_g,v_p]\cup v_pv_{i_j}\cup P[v_{h+2},v_{i_j}]\cup v_{h+2}v_1\cup P[v_1,v_h]$$ is
longer than $C$ unless $s=2$ and $g=i_{j+1}-1$. In this case, by~\eqref{m7}, the cycle
$$v_hv_g\cup P[v_g,v_{i_0}]\cup v_{i_0}v_p\cup P[v_p,v_{g+1}]\cup v_{g+1}v_1 \cup P[v_1,v_{h-2}]\cup v_{h-2}v'_1\cup v'_1v_{i_0-2}
\cup P[v_h,v_{i_0-2}]$$
has at least $2k$ vertices, a contradiction. So, suppose $g>i_q$. Then the cycle
$$C_4=Q\cup P[v_g,v_p]\cup v_pv_{2\lfloor (g-1)/2\rfloor}\cup P[v_{2\lfloor (g-1)/2\rfloor},v_{h+2}]\cup v_{h+2}v_1\cup
P[v_1,v_h]$$
has more than $2k-2$ vertices, a contradiction. This proves~\eqref{m11}.

To finish the proof of Part 5 by contradiction, suppose that for some even $h\leq i_0-2$, vertex $v_h$ has a neighbor 
$u$ that has a neighbor $w\notin N_P(v_1)$. By~\eqref{m12},~\eqref{m12'} and~\eqref{m11}, 
 $u\notin V(P)+v'_1$.
Since $u$ is in the same partite set of $G$ as $v_1$,~\eqref{m11} implies that $w\notin V(P)$. Since $G$ is $2$-connected, $G-u$
has a path $Q$ connecting $w$ with $V(P)+v'_1$ internally disjoint from $P+v_1'$. Let $Q=w_1,\ldots,w_s$, where $w_1=w$ and
either $w_s=v'_1$ or $w_s=v_\ell\in V(P)$. By~\eqref{m12} and~\eqref{m12'}, 
$w_s\notin \{v_1,v_3,\ldots,v_{i_0-1},v'_1\}$. %, then the path $P'=v'_1Q_w\cup wu\cup uv_h\cup P[v_1,v_h]\cup v_1v_{h+2}
%\cup P[v_{h+2},v_p]$ starts from $v'_1\in V(G_\alpha)$, ends with the same $v_p$, and is longer than $P$, a contradiction. 
So, in view of~\eqref{m11},
$w_s=v_\ell\in V(P)$, where $\ell\in \{2,4,\ldots,i_0-2\}\cup \{i_0,\ldots,i_q\}$. If $\ell\in \cup \{i_1,\ldots,i_q\}$, say $\ell = i_j$ then the cycle
$$C_5=v_hu\cup uw\cup Q\cup P[v_{i_j},v_p]\cup v_pv_{i_{j-1}}\cup P[v_{i_{j-1}},v_{h+2}]\cup v_{h+2}v_1\cup P[v_1,v_h]$$ is longer than $C$.
The last possibility is that $1\leq g\leq i_0$. Since $G$ is $2$-connected, we may assume that $g\neq h$ (indeed, if $g=h$, then
$G-v_h$ has a path from $V(Q)-v_g+u$ to $V(P)+v_1'$ which together with a part of $Q$ can play the role of $Q$). For definiteness, suppose
$g>h$ (the case of $g<h$ works the same way with the roles of $v_h$ and $v_g$ switched). If $h<g\leq h+2$, then the path
$P[v_1,v_h]\cup v_hu\cup uw \cup Q\cup P[v_g,v_p]$ has the same ends as $P$, but is longer. Let $g\geq h+3$.
Then the path 
$$v_1v_{2\lfloor (g-1)/2\rfloor}\cup P[v_{2\lfloor (g-1)/2\rfloor},v_{2\lceil (h+1)/2\rceil}]\cup v_{2\lceil (h+1)/2\rceil}v'_1\cup v'_1v_2
\cup P[v_2,v_h]\cup v_hu\cup uw\cup Q\cup P[v_g,v_p]$$
has the same ends as $P$, but is longer, a contradiction. Similarly, we obtain the symmetric part of Part 5.

Finally, we will show Part 2 of the definition of saturated crossing formation. Suppose there exists some odd $h \leq i_0 - 1$ such that for some $s \in \{i_0, \ldots, i_q\} \cup  \{2, 4, \ldots, i_0-2\}$, $v_hv_s \notin E(G)$. By Part 5, $N(v_h) \subset N(v_1)=N(v_1')$. Also, $v_{h-1} v_1', v_{h+1}v_1' \in E(G)$ (by \eqref{m10'}), so we can replace $v_h$ in $P$ with $v_1'$ to obtain a new path such that $\sum_{i=1}^p d_P(v_i) < \sum_{i=1}^p d_P(v_i) - d_P(v_h) + d_{P}(v_1')$, a contradiction. Together with the symmetric argument for the other side of $P$, we have shown that Part 2 of the definition of saturated crossing formation holds.
\end{proof}

Let $P$ be a path satisfying the conditions of Theorem~\ref{cycle1}. For simplicity, we denote \[P = L \cup H_1 \cup ... \cup H_q \cup R\] where $L = P[v_1, v_{i_0}]$, and $R = P[v_{i_q}, v_p]$, $V(H_1) \cap V(L) = \{v_{i_0}\}, V(H_t) \cap V(H_{t+1})= \{v_{i_t}\}$ for all $1\leq t \leq q-1$, and $V(H_{q}) \cap R = \{v_{i_q}\}$.  Let $H := H_1 \cup \ldots \cup H_q$.

\begin{lem}\label{sep}Let $P$ satisfy the conditions of Theorem~\ref{cycle1}. Let $I:= \{v_{i_0}, \ldots, v_{i_q}\}$, $L^o := L - I, R^o:= R - I, H^o := H - I$. Then $I$ separates $L^o, R^o,$ and $H^o$. That is, $L^o, R^o,$ and $H^o$ are each in different connected components in $G - I$. 
\end{lem}

\begin{proof}
Let $Q=z_1,z_2, \ldots, z_s$ be a shortest path that between vertices from two different sets in $\{L^o, R^o, H^o\}$. By minimality, $Q$ only intersects $P$ at $z_1$ and $z_s$. Also, $|V(Q)| \geq 3$ by Part 5 of the definition of saturated crossing formation.

Without loss of generality, $z_1 \in L^o$ (so $z_s \in H^o \cup R^o$). (Note that the case where $Q$ goes from $R^o$ to $H^o$ is symmetric to the case from $L^o$ to $H^o$.) By Part 5, since odd vertices in $P$ only have neighbors in $P$, $z_1$ and $z_s$ must be even. Also by Part 5, $N(z_2) \subseteq N(v_1) \subseteq L \cup X$. In particular, $z_3$ is in $P$, so we must have $z_3 = z_s$, but $z_s \in L \cup I$, where $(L \cup I) \cap (R^o \cup H^o) = \emptyset$, a contradiction.
\end{proof} 

\begin{claim}\label{dispath} Under  the conditions of Theorem~\ref{cycle1}, for any $0\leq s < t \leq q$, let $Q$ be a $(v_{i_s}, v_{i_t})$-path that is internally disjoint from $P$. Then
\begin{enumerate}
\item if $P$ has exactly one pair of crossing neighbors (so $s = 0, t = 1$), then $|V(Q)| < k+1$.
\item if $P$ has multiple pairs of crossing neighbors, then $|V(Q)| < 6$. 
\end{enumerate} \end{claim}

\begin{proof}
First suppose $P$ has only one pair of crossing neighbors. Then $v_1$ has $\alpha + 1 - 1$ neighbors in $L$. That is, $|V(L)| \geq 2\alpha \geq k-1$. If $|V(Q)| \geq k+1$, then the cycle $P[v_1, v_{i_0}] \cup Q \cup v_{i_1}v_1$ has length at least $k-1 + k+1 - 1 = 2k-1$, a contradiction.

Otherwise, if $P$ has more than one pair of crossing neighbors, then each pair has 3 vertices strictly between them in $P$. Suppose $|V(Q)| \geq 6$ (so there are at least 4 internal vertices). If $t = s+1$, then replacing $P[v_{i_s}, v_{i_{s+1}}]$ with $Q$ gives a longer path with the same endpoints as $P$. So we may assume $s \geq t+2$. Then the cycle \[P[v_1, v_{i_s}] \cup Q \cup P[v_{i_t}, v_p] \cup v_pv_{i_{t-1}} \cup P[v_{i_{t-1}}, v_{i_{s+1}}] \cup v_{i_{s+1}}v_1\] has length at least $|V(P)| + 2$.   
\end{proof}

Observe that because $P$ is in crossing formation, $|V(L)| , |V(R)| \geq 4$ and $c(G) =2(k-1) \geq |V(L)| + |V(R)| = 8$, thus $k \geq 5$.

\section{The Main Lemma}

%For the rest of the paper, we assign a {\em deficiency} to vertices with degree less than $r$.

%\begin{definition}
Recall that the {\em deficiency} of a vertex $x\in X$ in a bipartite graph
 $G = (X, Y; E)$ is $D_G(x) := \max\{0, r - d_G(x)\}$. For a subset $X^* \subseteq X$, the {\em deficiency} of $X^*$ as $D(G, X^*) := \sum_{x \in X^*} D_G(x)$. 
%\end{definition}

Our goal is to eventually to prove the Main Theorem, Theorem~\ref{main}.

%{\bf Theorem~\ref{main}} (Main Theorem for bipartite graphs){\bf.}
%{\em Let $ k\geq 4$, $r\geq k+1$  and $m,m^*,n$ be positive integers with $n\geq k$, $m \geq m^*\geq  k-1$ and $m\geq k$. 
%Let $G=(X,Y;E)$ be a bipartite $2$-connected  graph with parts $X$ and $Y$, where $|X|=m$, $|Y|=n$, and
%let $X^*\subseteq X$ with $|X^*|=m^*$.  If $c(G)<2k$, then  \begin{equation}\label{e2}
 %m^*\leq  \frac{k}{2r-k+2} (n-1 + D(G, X^*)).
%\end{equation}
%or $r = k$, $n \leq k-1$, and $m^* \leq n-k+1 + D(G, X^*)$.
%}

The first big step is to prove the Main Lemma below that states roughly that graphs that contain a path in saturated crossing formation satisfy Theorem~\ref{main}.

\begin{lem}[Main Lemma]\label{mainlemma}Let $k \geq 5$ be odd, and let $G= (X,Y;E)$ and $X^* \subseteq X$ be a minimum (with respect to $|X|$) counterexample to Theorem~\ref{main}. Fix any $X^* \subseteq X$ and set $Y = Y^*$. If $|Y| \geq k$ and $P$ is a path as in the hypothesis of Theorem~\ref{cycle1}, then $P$ is not in saturated crossing formation. 
\end{lem}

\subsection{Lemmas for induction}
We first prove a series of lemmas. Often, we will use the following inductive argument:

\begin{lem}\label{2connind}Let $k \geq 4$. Let $G=(X,Y; E)$ and $X^* \subseteq X$ be a minimum (with respect to $|X|$) counterexample to Theorem~\ref{main}. Suppose $|X| \geq k+1$, $|X^*| \geq k$, $|Y| \geq k$ and there exists a vertex $x \in X^*$  with $d(x) \leq k-2$. Then $G - x$ is not 2-connected.
\end{lem}

\begin{proof}Suppose $G - x$ is 2-connected. 
%Add edges to $G - x$ until it is $2k$-saturated. Note that adding edges to $G - x$ cannot increase the deficiency.
%That is, for all vertices $v$ in $G - x$, $D_{G - x}(v) \geq D_{G'}(v)$. Thus \[D(G', X^* - x) \leq D(G-x, X^* - x) = D(G, X^* - x) - D_G(x).\]
As $d_G(x) \leq k-2$, we have $D_G(x) \geq r - k+2$. Since $|X - x| \geq k+1 - 1 = k$ and $|X^* - x| \geq k - 1$, by the choice of $G$ as a minimum counterexample, $G-x$ and $X^* - x$ satisfy
\[|X^*| - 1 = |X^* - x| \leq \frac{k}{2r-k+2}(n-1 + D(G', X^*- x)) \leq \frac{k}{2r-k+2}(n-1+D(G, X^*)) - \frac{k(r-k+2)}{2r-k+2}.\]
Elementary calculation shows that $\frac{k}{2r-k+2}(r-k+2) \geq 1$ whenever $r \geq k-1$. Thus \[|X^*| \leq \frac{k}{2r-k+2}(n-1+D(G, X^*)),\] a contradiction.

\end{proof}
\begin{lem}\label{ind}
Let $k \geq 4$. Let $G=(X,Y; E)$ and $X^* \subseteq X$ be a minimum (with respect to $|X|$) counterexample to Theorem~\ref{main}. Suppose $|X| \geq k+1$, $|X^*| \geq k$, $|Y| \geq k$, and $P$ is a path in $G$ with an endpoint $x$. Suppose also that $x$ has no neighbors outside of $P$, $d(x) \leq k-2$, and $x \in X^*$.
Then there does not exist a vertex $x' \in V(G) - V(P)$ such that $N(x) \subseteq N(x')$.
\end{lem}

\begin{proof}Suppose such a vertex $x'$ exists. If $G - x$ is not 2-connected, then it contains a cut vertex $v$ such that $(G - x) - v$ contains at least two components, $C_1$ and $C_2$, and $v$ is the only vertex in $G - x$ with neighbors in both $C_1$ and $C_2$. Then in $G$, $x$ and $v$ form a cut set, and $x$ and $v$ are the only vertices in $G$ with neighbors in both $C_1$ and $C_2$. As $N(x) \subseteq N(x')$,  $v = x'$. Let $y_1$, and $y_2$ be neighbors of $x$ such that $y_1 \in C_1$, and $y_2 \in C_2$. Because $N(x) \subseteq V(P)$, $y_1, y_2 \in V(P) - x$, but the path $P[y_1, y_2]$ is a $(y_1, y_2)$-path in $G$ that avoids both $x$ and $x'$, a contradiction.

\end{proof}

\subsection{Paths in saturated crossing formation}

\begin{lem}\label{satcross}Let $k \geq 5$ be odd, and let $G= (X,Y;E)$ and $X^* \subseteq X$ be a minimum (with respect to $|X|$) counterexample to Theorem~\ref{main}. Fix any $X^*$ and set $Y = Y^*$. If $|Y| \geq k$ and $P=v_1, \ldots, v_p$ is a path as in the hypothesis of Theorem~\ref{cycle1}, then the endpoints $v_1$ and $v_p$ of $P$ belong to the partite set $Y$ of $G$.  \end{lem}

\begin{proof}
Suppose $v_1, v_p \in X$. By Lemma~\ref{cycle}, one of the endpoints of $P$, say $v_p$, must satisfy $d_P(v_p) \leq \frac{k+1}{2}$. Since $v_2 \in G_\alpha(X^*) \subseteq X^*$, $v_{p-1} \in G_{k-1 - \alpha}(X^*) \subseteq X^*$ and $v_2$ and $ v_{p-1}$ have no common neighbors by Lemma~\ref{sep}, we have $|X^*| \geq d_{G_\alpha}(v_2) + d_{G_{k-1-\alpha}}(v_{p-1}) \geq \alpha + 1 + k - \alpha=k+1$. Also, by Part 4 of the definition of saturated crossing formation, there exists a vertex $v_p' \in V(G) - V(P)$ with $N_P(v_p') = N_P(v_p)$.  By Part 5 of the definition of saturated crossing formation, $N(v_p) = N_P(v_p)$, so $d(v_P) \leq k-2$. But the existence of $v_p'$ contradicts Lemma~\ref{ind}.
\end{proof}

Suppose $P = v_1, \ldots, v_p$ is in saturated crossing position. Denote $P = L \cup H_1 \cup ... \cup H_\ell \cup R$ as before.

\begin{lem}\label{empty} Under  the conditions of Theorem~\ref{cycle1}, let $F$ be a component of $G - \{v_{i_0}, \ldots, v_{i_q}\}$ distinct from the components containing $L$ and $R$. Then the $\frac{k-1}{2}(X^* \cap F, Y \cap F)$-disintegration of $F$ is empty.\end{lem}

\begin{proof}
Set $\alpha' = (k-1)/2$ and denote $F_{\alpha'}= G_{\alpha'}(X^*\cap F, Y \cap F) = G_{k-1-\alpha'}(X^*\cap F, Y \cap F)$. Because $G$ is 2-connected, there are at least 2 neighbors of $F$ in $P$, and so these neighbors must be contained in $\{v_{i_0}, \ldots, v_{i_q}\}$ by Lemma~\ref{sep}. 
%Note that because $v_2$ is not contained in a crossing pair (otherwise $v_2v_p \in E(G)$), $v_1$ has at most $q+1 - 1$ neighbors in $\{v_{i_0}, \ldots, v_{i_q}\}$. Thus $q \leq \alpha - 1$. Similarly, in view of $v_p$, $q \leq k -2 - \alpha$. Thus $q \leq (k-3)/2$, and $|V(P[v_{i_0}, v_{i_q}])|\leq 4((k-3)/2) + 1 = 2(k-3) + 1$. 

If $F_{\alpha'}$ is complete bipartite, then each part has size at least $\alpha' + 1 = (k+1)/2$, so we may find a path of length at least $k+1$ from some $v_{i_s}$ to some $v_{i_t}$ whose internal vertices are all from $F$, violating Lemma~\ref{dispath}.

If $F_{\alpha'}$ is not complete bipartite, then fix a longest path $P_F=u_1, \ldots, u_{p'}$ with nonadjacent endpoints in $F_{\alpha'}$ such that $\sum_{i=1}^{p'} d_{P_F}(u_i)$ is maximized. Then by Theorem~\ref{cycle1}, $P_F$ must be in saturated crossing formation. Again, $u_1$ has exactly $\alpha' + 1$ neighbors in $F_{\alpha'}$ in $P_F$. Furthermore, by Lemma~\ref{satcross}, $u_1, u_{p'} \in Y$. 

Denote $P_F = L' \cup H'_1 \cup \ldots \cup H'_{q'} \cup R'$ where $H'_1 \cap L' = \{u_{j_0}\}$, for each $0 \leq s \leq q' - 1, H'_{s} \cap H'_{s+1} = u_{j_s},$ and $H'_{q'} \cap R' = \{u_{j_{q'}}\}$. There exists a cycle $C' = P_F[u_1, u_{j_0}] \cup u_{j_0} u_{p'} \cup P_F[u_{p'}, u_{j_1}] \cup u_{j_1}u_1$ which has length exactly $2(k-1)$. 

{\bf Case 1:} at most 1 vertex from $\{v_{i_0}, \ldots, v_{i_q}\}$ is contained in $C'$. Then because $\{v_{i_0}, \ldots, v_{i_q}\}$ separates $F$ from $G - F$, $C'$ never leaves $F \cup \{v_{i_0}, \ldots, v_{i_q}\}$. Choose two shortest disjoint paths $P_s, P_t$ from $\{v_{i_0}, \ldots, v_{i_q}\}$ to $V(C')$ (possibly $P_s$ or $P_t$ may be a single vertex). Such paths exist because $G$ is 2-connected. Furthermore, by choice of $P_s$ and $P_t$, the paths each contain exactly one vertex from $\{v_{i_0}, \ldots, v_{i_q}\}$ and one vertex from $C'$, and hence the paths cannot leave $F \cup \{v_{i_0}, \ldots, v_{i_q}\}$. Say $P_s$ has endpoints $v_{i_s}$ and $u_{s'} \in V(C')$ and $P_t$ has endpoints $v_{i_t}$ and $u_{t'} \in V(C')$. 

Because $|V(C')| = 2(k-1)$, one of the $(u_{s'},u_{t'})$-paths along $C'$ must have at least $k-1$ edges, i.e., $k$ vertices. Then because at least one of $P_s$ or $P_t$ has at least 2 vertices by the case, we have that $P_s \cup P_F[u_{s'}, u_{t'}] \cup P_t$ is a path of length at least $k+1$ from $v_{i_s}$ to $v_{i_t}$, contradicting Lemma~\ref{dispath}.

{\bf Case 2:} at least 2 vertices from $\{v_{i_0}, \ldots, v_{i_q}\}$ are contained in $C'$. Let $N_F = \{v_{i_0}, \ldots, v_{i_q}\} \cap V(P_F)$. If any vertex $v_{i_s} \in N_F$ appears in $L'$ or $R'$, then because $v_{i_s}$ is even, we have that $v_{i_s} \in N(u_1) \cup N(u_{p'})$. Therefore either $d_{P_F}(u_1) \geq (k+1)/2 + 1$ or $d_{P_F}(u_{p'}) \geq (k+1)/2 + 1$, which would give us a longer cycle by Lemma~\ref{cycle}, a contradiction. Therefore we may assume that each vertex in $N_F$ appears in $P_F$ strictly between some crossing neighbors. We will show that $P_F$ has only one pair of crossing neighbors, in which case $N_F \cap V(C') = \emptyset$, leading to a contradiction.

Suppose not, then each pair of crossing neighbors have exactly 3 vertices strictly between them in $P_F$. Because each $v_{i_s} \in N_F$ is even, $v_{i_s}$ must appear as the middle vertex between a pair of crossing neighbors, say $u_{j_{s'-1}}$ and $u_{j_{s
'}}$, and there cannot be any other vertices from $N_F$ in between these crossing neighbors. Furthermore, the predecessor and the successor of $v_{i_s}$ in $P_F$ belong to $K$ since they are odd neighbors of $u_{j_{s'_1}}$ or $u_{j_{s'}}$ which are both in $F$. Thus $P_F$ never leaves $K \cup \{v_{i_0}, \ldots, v_{i_q}\}$. 

Suppose there exists $v_{i_s}, v_{i_t} \in N_F$ such that $v_{i_s}$ and $v_{i_t}$ appear between consecutive pairs of crossing neighbors in $P_F$. Say $v_{i_s} \in H'_{s'}$ and $v_{i_t} \in H'_{s'+1}$. Then the cycle $C''= P_F[u_1, v_{i_s}] \cup [v_1, v_2, v_3] \cup P_F[v_{i_t}, u_{P'}]  \cup u_{p'}u_{j_{s'}} \cup u_{j_{s'}}u_1$ omits only the successor of $v_{i_s}$ and the predecessor of $v_{i_t}$ in $P_F$ and includes three additional vertices, $v_1, v_2, v_3$ from $P$. Therefore $|V(C'')| = |V(P_K)| - 2 + 3 >2k$, a contradiction.
Otherwise, each $v_{i_s}, v_{i_t} \in N_F$ appear in $P_F$ between nonconsecutive pairs of crossing neighbors of $P_F$. Pick $v_{i_s}, v_{i_t}$ such that no other vertex in $N_F$ lies between them in $P_F$. Then $P_F[v_{i_s}, v_{i_t}]$ is a path with at least $9$ vertices that is internally disjoint from $P$, contradicting Lemma~\ref{dispath}.
It follows that $F_{\alpha'}$ is empty.
\end{proof}

We are now ready to prove the Main Lemma.

{\bf Lemma~\ref{mainlemma}.} {\em Let $k \geq 5$ be odd, and let $G= (X,Y;E)$ and $X^* \subseteq X$ be a minimum (with respect to $|X|$) counterexample to Theorem~\ref{main}. Fix any $X^* \subseteq X$ and set $Y = Y^*$. If $|Y| \geq k$ and $P$ is a path as in the hypothesis of Theorem~\ref{cycle1}, then $P$ is not in saturated crossing formation. }

 \begin{proof} Let $P = v_1, \ldots, v_p$ be the path in saturated crossing formation. Let $C_L$ and $C_R$ denote the connected components of $G - \{v_{i_0}, \ldots, v_{i_q}\}$ that contain $L - \{v_{i_0}\}$ and $R - \{v_{i_q}\}$ respectively. By Lemma~\ref{sep}, $C_L$ and $C_R$ are distinct. Let $D = C_L \cup C_R \cup \{v_{i_0}, \ldots, v_{i_q}\}$ and set $X' = X^* \cap D$, $n' = |Y \cap D|.$

By Lemma~\ref{satcross}, $v_1, v_p \in Y$, and hence each odd vertex in $P$ also belongs to $Y$.  We first claim that there cannot be any $X$ vertices in $C_L$ outside of $P$: suppose such vertices exist, and pick a shortest path $Q$ from $X-L$ to $P$ with endpoints $x \in X - L$, $v \in L$. If $v$ is odd, then $Q=vx$, but $x$ is a neighbor of $v$ outside of $P$, violating Part 5 of the definition of saturated crossing formation. If $v$ is even, then $Q$ contains at least 3 vertices, and the predecessor of $v$ in $Q$ is in $X$ and has a neighbor outside of $P$, again violating Part 5. Therefore $X \cap C_L \subseteq L$. Similarly, $X \cap C_R \subseteq R$. It follows from Part 2 of the definition of saturated crossing formations that $X \cap D \subseteq N(v_1) \cup N(v_p)$.

Observe that we must have $X' = N(v_1) \cup N(v_p)$: by definition of saturated crossing formation, all neighbors of $v_1$ and all neighbors of $v_p$ belong in $G_\alpha \cup G_{k-1-\alpha} \subseteq X^* \cap D = X'$. Furthermore, $N(v_1) \cup N(v_p)$ contains all of the even vertices in $L \cup R \cup \{v_{i_0}, \ldots, v_{i_q}\}$ and therefore all of the $X$ vertices in $D$. This proves that $X' = N(v_1) \cup N(v_p)$.

As $v_1$ and $v_p$ share at least two neighbors (the crossing neighbors), we have \[|X'| = |N(v_1) \cup N(v_p)| \leq (\alpha + 1) + (k - \alpha) - 2 = k-1.\]

Furthermore, because $v_2$ and $v_{p-1}$ share no neighbors, \[n' + D(G, X') \geq d_G(v_2) + D_G(v_2) + d_G(v_{p-1}) + D_G(v_{p-1}) \geq 2r.\] 

Putting these together, we have
\[\frac{|X'|}{n'-1+D(G, X')} \leq \frac{k-1}{2r-1},\] and therefore
\begin{equation}\label{LR}
|X'| \leq \frac{k-1}{2r-1}(n'-1+D(G,X'))  < \frac{k}{2r-k+2}(n' - 1 + D(G, X')).
\end{equation} 

Finally, for any component $F$ of $G - \{v_{i_0}, \ldots v_{i_q}\}$ distinct from $C_L$ and $C_R$, we have that the $\frac{k-1}{2}(X^* \cap F, Y\cap F)$-disintegration of $F$ is empty by Lemma~\ref{empty}. Set $n_F = |Y \cap F|$. Each time we delete a vertex in the disintegration process, we delete at most $(k-1)/2$ edges until we reach the last $k-1$ vertices where there are at most $((k-1)/2)^2$ edges. Thus \[e(F) \leq \frac{k-1}{2}(n_F+m_F - (k-1)) + \left(\frac{k-1}{2}\right)^2 = \frac{k-1}{2}(n_F+m_F - \frac{k-1}{2}).\] As $e(F) \geq r |X^* \cap F| - D(G, (X^* \cap F))$, we have

\begin{eqnarray}\label{ml2}
|X^* \cap F| \leq \frac{\frac{k-1}{2}\left(n_F-\frac{k-1}{2}\right) + D(G, X^* \cap F)}{r - \frac{k-1}{2}} < \frac{k}{2r-k+2}(n_F + D(G, X^* \cap F)).
\end{eqnarray}

Combining \eqref{LR} and \eqref{ml2},

\begin{eqnarray*}
|X^*| &=& |X'| + \sum_{F\neq C_L, C_R} |X^* \cap F|
\\ &< &  \frac{k}{2r-k+2}(n'-1 + D(G, X')) + \sum_{F \neq C_L, C_R} \frac{k}{2r-k+2}\left(n_F +D(G, X^* \cap F) \right)\\
&\leq & \frac{k}{2r-k+2}(n-1 + D(G, X^*)),
\end{eqnarray*} a contradiction.
\end{proof}

\section{Large complete bipartite subgraphs in extremal graphs}

We will need three more lemmas to be used later in the Proof of Theorem~\ref{main}.

\begin{definition}\label{uupath}For a set $U$ of vertices in a graph $G$, we say a {\em $U,U$-path} is a path whose ends are in $U$ and all internal vertices are not in $U$.
\end{definition}

We will use several times the following simple property of $2$-connected graphs.
\begin{proper}\label{pr1}
Let $G$ be a $2$-connected graph, $U\subset V(G)$ with $|U|\geq 2$, and $xy$ be an edge in $E(G)$ such that
$\{x,y\}\not\subseteq U$. Then there is a $U,U$-path $P$  containing $xy$.
\end{proper}

%\section{$c_1+c_2$}
%A bipartite graph $G$ is {\em $k$-saturated}, if $c(G)<2k$, but after adding to $G$ any edge $xy$ connecting distinct parts
%such that $xy\notin E(G)$, the new graph has a cycle of length at least $2k$.

\begin{lem}\label{c1} Let $m,n\geq k\geq 4$ be positive integers.
Let $G=(X,Y;E)$ be a  bipartite $2$-connected graph with $|X|=m$, $|Y|=n$ and $c(G)<2k$.
Suppose $G$   contains a copy $K$ of $K_{k-1,k-2}$ with parts $A\subset X$ and $B\subset Y$ such that 
$|A|=k-1$. Then 
\begin{equation}\label{le1}
| N(Y-B)|=2\quad\mbox{or}\quad |N(Y-B)\cap A|\leq 1.
\end{equation}
\end{lem}

{\em Proof.} Suppose that $G=(X,Y;E)$ is a  bipartite $2$-connected 
 graph with $|X|=m\geq k$, $|Y|=n\geq k$ and $c(G)<2k$ containing a copy $K=(A,B;E_1)$ of 
$K_{k-1,k-2}$ with $|A|=k-1$, $|B|=k-2$, $A\subset X$ and $B\subset Y$. Suppose further that~(\ref{le1}) does not hold, i.e., that
\[
| N(Y-B)|\geq 3\quad\text{and}\quad |N(Y-B)\cap A|\geq 2.\]

First, we remark that
\begin{equation}\label{1n}
  \mbox{\em each $A\cup B, A\cup B$-path in $G$  contains at most one vertex in $Y-B$.}
\end{equation}
Indeed, if an $A\cup B, A\cup B$-path $P$ contains two vertices in $Y-B$, then $G[A\cup B\cup V(P)]$ has a cycle $C$ that
contains $B\cup V(P)$. This $C$ has at least $k$ vertices in $Y$, and hence $|C|\geq 2k$, contradicting $c(G)<2k$.
 This proves~(\ref{1n}).

\medskip
{\bf Case 1:} There is $y_1\in Y-B$ with $|N(y_1)\cap A|\geq 2$. Suppose $N(y_1)\cap A=\{a_1,\ldots,a_q\}$.

{\bf Case 1.1:} $V(G)-A-B-y_1$ has an edge $xy_2$. By Property~\ref{pr1}, there is an $(A\cup B,A\cup B)$-path $P$
containing $xy_1$. Let $P=w_1w_2\ldots w_h$, $x=w_j$ and $y_1=w_{j+1}$ for some $2\leq j\leq h-2$.
By~(\ref{1n}), $y_1\notin P$ and $j=2$. In particular, $w_1\in B$. Let $P'=a_1y_1a_2$. Then
$G[A\cup B\cup V(P)\cup \{y_1\}]$ has a cycle $C$ containing $B\cup P\cup P'$ and hence at least 
$k$ vertices in $Y$. So $|C|\geq 2k$, contradicting $c(G)<2k$.

{\bf Case 1.2:} $V(G)-A-B-y_1$ is an independent set. Then any $y_2\in Y-A-y_1$ has at least two neighbors in $A$. So by the Case 1.1 for $y_2$ in place
of 
$y_1$, $V(G)-A-B-y_2$ is an independent set. Then $V(G)-A-B$ is an independent set. In particular, any $x\in X-A$ has two neighbors in $B$. So, the graph
$G[A\cup B\cup \{y_1,y_2,x\}]$ has no cycle containing $B\cup \{y_1,y_2\}$ only if $q=2$ and $N(y_2)\cap A=\{a_1,a_2\}$. Trying each
$y\in Y-B-y_1$ as $y_2$, we conclude that $\bigcup_{y\in Y}N(y)=\{a_1,a_2\}$, so~\eqref{le1} holds.

\medskip
{\bf Case 2:}  $|N(y)\cap A|\leq 1$ for every $y\in Y-B$. Because $|N(Y-B) \cap A| \geq 2$, there are distinct $a_1,a_2\in A$ adjacent to $Y-B$. Let $a_1y_1,a_2y_2\in E(G)$ where $y_1,y_2\in Y-B$.
By the case, $y_2\neq y_1$. By Property~\ref{pr1}, for $j=1,2$ there is an $(A\cup B,A\cup B)$-path 
$P_j=w_{1,j},w_{2,j},\ldots, w_{h_j,j}$
containing $a_jy_j$. Since $a_j\in A$, we may assume $w_{1,j}=a_j$ and $w_{2,j}=y_j$. By the case $w_{3,j}\notin A$, and
so $h_j\geq 4$. Furthermore, by~(\ref{1n}), $w_{4,j}\in B$, and so $h_1=h_2=4$. If $w_{3,2}=w_{3,1}$, then the path
$a_1,y_1,w_{3,1},y_2,a_2$ contradicts~(\ref{1n}). So, paths $P_1$ and $P_2$ are internally disjoint and have at most one
common end. Thus $G[A\cup B\cup V(P_1)\cup V(P_2)]$ has a cycle $C$ containing $B\cup  V(P_1)\cup V(P_2)$, which
implies $|C|\geq 2k$, contradicting $c(G)<2k$.
\qed

\begin{lem}\label{pcover} Let $H$  be a bipartite graph with parts $A$ and $B$, where $|B| = g\geq 2$. Suppose $H$ has no isolated vertices
 and for each $b \in B$, $d(b) \geq g$. 
Then either (i) $H = K_{g,g}$ or (ii) there exist disjoint paths $Q_1, \ldots, Q_\ell$ such that for each $1\leq i\leq \ell$, $Q_i$ has both ends in $A$, 
and $B \subset V(Q_1 \cup \ldots \cup Q_\ell)$. 
\end{lem}

\begin{proof}
We proceed by induction. If $g = 2$ and $H \neq K_{2,2}$, then $H$ contains either a $P_5$ or two disjoint copies of $P_3$, both of which satisfy (ii). 
Now let $g > 2$. Fix $ab \in E(G)$ such that $a \in A, b \in B$. Set $B' = B -\{b\}$ and $A' = N(B') - \{a\}$. Then $H' = H[B' \cup A']$ satisfies the 
conditions of the lemma
%induction  hypothesis
 for $g-1$.

Suppose first that $H' = K_{g-1, g-1}$. Then because each vertex $b' \in B'$ has exactly $g-1$ neighbors in $H'$ and at least $g$ neighbors in $H$,
 $b'a \in E(H)$ for each $b'\in B'$. 
If $b$ has no  neighbors outside $A' \cup \{a\}$, then $G = K_{a, a}$. Otherwise, if $b$ has a neighbor $a' \in A-A' - \{a\}$, 
we may take any path $P$ with $2g$ vertices starting with $a$ and covering  $A' \cup B'$  and append the edge $ba'$ to $P$.

If $H' \neq K_{g-1, g-1}$, then let $Q'_1, \ldots, Q'_{q}$ be the set of paths satisfying (ii) for $H'$.
 If $b$ has a neighbor $a' \in A - \{a\}-\bigcup_{i=1}^q V(Q'_i)$, then we take the set of paths $Q'_1, \ldots, Q'_{q}, aba'$.
 Otherwise, all neighbors of $b$ are in $N(B') +a$. In particular, $b$ has at least $g-1$ neighbors distinct from $a$. 
But each $Q'_i$ has fewer internal vertices in $A$ than in $B$.
Thus paths $Q'_1, \ldots, Q'_{q}$ together have at most $g-2$ internal vertices in $A$. 
 Thus $b$ has a neighbor $a'$ that is an end of a path, say of $Q'_{q}$. Then we append the path $a',b,a$ to $Q'_{q}$.
\end{proof}

\begin{lem}\label{nl2}Let $G=(X,Y;E)$ and $X^* \subseteq X$ be a counterexample to Theorem~\ref{main} with minimum $|X|$.
Then $G$ cannot contain a complete bipartite subgraph $G'=K_{s,t}$ with parts $A\subseteq X^*$ and $B\subseteq Y$ such that
\begin{equation}\label{j14}
 \mbox{$|A|=t\geq k$ and $|B|=s$ with $k/2 \leq s \leq k-2$.}
\end{equation}
\end{lem}

{\em Proof.} Suppose that such a $K_{s,t}$ exists. We may assume that $s$ and $t$ are largest possible, i.e. each $x\in X-A$ has a nonneighbor in $B$ and
 each $y\in Y-B$ has a nonneighbor in $A$.
 
% Since $c(G)\leq 2k-2$, $s\leq k-1$. Moreover, suppose $s=k-1$.
%Let $a_1\in A$.
%Since $d(a_1)\geq r$, there is $z\in N(a_1)-B$. Since $G$ is $2$-connected, $G-a_1$ contains a path $Q$ from $z$
%to $A\cup B-a_1$. So, if $s=k-1$, then $G[A\cup B\cup V(Q)]$ contains a cycle of length at least $2k$. Thus we may assume
%that
%\begin{equation}\label{j91'}
%\mbox{$s\leq k-2$ (which yields $k\geq 4$, since $s\geq k/2$).}
%\end{equation}

Consider a mixed $(k-s,k-s-1)(X^*, Y)$-disintegration of $G$: we first delete all vertices from $X - X^*$ and then consecutively delete remaining vertices in $X$ if their degrees in the current
graph are
at most $k-s$ and vertices in $Y-B$ if their degrees in the current graph are at most $k-s-1$\footnote{Note that we do not delete
vertices in $A$ even when $s=k/2$ and they have degree $k/2=k-s$ in the current graph.}. Let $G_0$ be the resulting graph. If
$G_0=G'=K_{s,t}$, then by~\eqref{j14},
\begin{equation}\label{j102}
rm^* - D(G, X^*) \leq e(G) \leq st+(m^*-t)(k-s)+(n-s)(k-s-1)=(2s-k)t+m^*(k-s)+(n-s)(k-s-1)
\end{equation}
$$\leq ((2s-k)+(k-s))m^*+(n-s)(k-s-1)=sm^*+(n-s)(k-s-1),$$
and hence \[m^*\leq\frac{(k-1-s)(n-s) + D(G, X^*)}{r-s}<
\frac{(k-1-s)(n-1 + D(G, X^*))}{r-k+s} \]\[\leq \frac{(k-1-(k/2)) (n-1+D(G, X^*)}{r-(k/2)} = \frac{(k-2)(n-1 + D(G, X^*)}{2r-k},\] but $\frac{k-2}{2r-k} < \frac{k}{2r-k+2}$, a contradiction. Thus, suppose $G_0\neq G''$, and the partite sets of $G_0$ are
$A\cup A'$ and $B\cup B'$.

Since $G$ is $2k$-saturated and $G_0$ is not complete bipartite, there exist paths with at least $2k$ vertices both ends of which are in $V(G_0)$ and
at least one end in $A\cup B$. Among such paths choose a path $P=v_1,\ldots,v_p$ with $v_1\in A\cup B$ so that\\
(P1) $p$ is maximum possible,\\
(P2) modulo (P1), $d_G(v_p)$ is maximum, and\\
(P3) modulo (P1) and (P2), $P$ has as many vertices from $A$ as possible.

\medskip
Our first observation is
\begin{equation}\label{j92}
v_1\in A.
\end{equation}
Indeed, if $v_1\in B$, then by (P1), each $a\in A$ is in $P$ and $d_P(v_p)\geq k-s$. Thus by~$t\geq k$ and $s\leq k-2$,
$$d_P(v_1)+d_P(v_p)\geq t+(k-s)\geq k+2.$$
 So by Lemma 1.2, $c(G)\geq 2k$, a contradiction.

By~\eqref{j92} and (P1), 
\begin{equation}\label{j92'}
B\subseteq V(P).
\end{equation}

\medskip
{\bf Case 1:} $d_{P}(v_p)\geq k-s+1$. Then Lemma 2.1 implies
\begin{equation}\label{j93}
d_P(v_1)=s, d_P(v_p)=k-s+1,\; \mbox{$p$ is odd, and $P$ has crossing neighbors $v_{i_1}$ and $v_{i_2}$.}
\end{equation}
Let $C=P[v_1,v_{i_1}]\cup v_{i_1}v_p\cup P[v_{i_2},v_p]\cup v_{i_2}v_1$. By the choice of $G$, $|C|\leq 2k-2$.
Since $N_P(v_1)^-\cap N_P(v_p)^+=\emptyset$, $|N_P(v_1)^-| + | N_P(v_p)^+|\geq k+1$, and
$C$ does not contain only two vertices from $N_P(v_1)^-\cup N_P(v_p)^+$,
\begin{equation}\label{j94}
|C|=2k-2\; \mbox{and each $v_i\in C\cap (A\cup A')$ is in $N_P(v_1)^-\cup N_P(v_p)^+$.}
\end{equation}
By~\eqref{j92'} and~\eqref{j93}, $N_P(v_1)=B$. In particular, $v_2\in B$. If $A\subset V(P)$, then by~\eqref{j14} and $s\leq k-2$, for the path
$P'=P-v_1$ we have $d_{P'}(v_2)+d_{P'}(v_p)\geq |A-v_1|+(k-s+1)\geq (t-1)+3\geq k+2$.
In this case,  by Lemma 1.2, $c(G)\geq 2k$, a contradiction. Thus, there is a vertex $a\in A-V(P)$. So, if
for some $3\leq i\leq p-2$, vertices $v_{i-1}$ and $v_{i+1}$ are in $B$, then the path $P''$ obtained from $P$ by replacing $v_i$
with $a$ has the same length and ends as $P$. Hence (P3) implies that
\begin{equation}\label{j95}
 \mbox{if for some $3\leq i\leq p-2$, vertices $v_{i-1}$ and $v_{i+1}$ are in $B$, then $v_i\in A$.}
\end{equation}

{\bf Case 1.1}: $v_1$ has no neighbors outside of $P$. Thus $D_G(v_1) = r-s \geq r-(k-2)$. As $v_1$ is contained in the $K_{s,t}$ and has no neighbors outside of $B$, it is easy to see that $G - v_1$ is 2-connected. Since $|V(P)| \geq 2k$ (and therefore $|V(P)| \geq 2k+1$ since $|V(P)|$ is odd), $|X - v_1| \geq k$. Furthermore, since $A \subseteq X^*$, $|X^*| \geq k$ and so $|X^* - v_1| \geq k-1$. Applying Lemma~\ref{2connind} yields a contradiction. 

{\bf Case 1.2}: $v_1$ has a neighbor $z \in N(v_1) -V(P)$.  Let $Q=w_1,\ldots,w_j$ be a path from $z=w_1$ to $P-v_1$ in
$G-v_1$. Suppose $w_j=v_h$.  Let $Q'=v_1z\cup Q$.
Since $z\notin V(P)$, $j\geq 2$. We claim that
\begin{equation}\label{j96}
 \mbox{for  $h-4\leq g\leq h-1$,  $v_{g}\notin N(v_p)$.}
\end{equation}
Indeed, otherwise the cycle $P[v_1,v_g]\cup v_gv_p\cup P[v_h,v_p]\cup Q'$ would have at least
$$2k+1-(h-g-1)+(j-1)\geq 2k+1-3+1=2k-1$$ vertices. This contradicts the choice of $G$.

Similarly, we show
\begin{equation}\label{j97}
 \mbox{  $v_{h}\notin P[v_{i_1+1},v_{i_2}]$.}
\end{equation}
Indeed, if $i_1+1\leq h\leq i_2$, then the cycle $P[v_1,v_{i_1}]\cup v_{i_1}v_p\cup P[v_h,v_p]\cup Q'$ would have at least
$|C|+1$ vertices, which means at least $2k$ vertices. 

Also
\begin{equation}\label{j98}
 \mbox{$\{v_{h-2},v_{h-1}\}\cap B=\emptyset$. Since $v_{i_2}\in N_P(v_1)=B$, this yields ${h}\notin \{{i_2+1},{i_2+2}\}$.}
\end{equation}
Indeed if $h-2\leq g\leq h-1$ and $v_g\in B$, then the path
$P[v_{g},v_1]\cup Q'\cup P[v_{h},v_p]$ starts from $v_{g}\in B$  and is longer than
$P$ (because if $g= h-2$ then by parity, $j\geq 3$).

Let $\alpha\in \{0,1\}$ be such that $h-2-\alpha$ is odd.
Then by~\eqref{j98}, $v_{h-2-\alpha}\notin B^-=N_P(v_1)^-$, 
by~\eqref{j96}, $v_{h-2-\alpha}\notin N_P(v_p)^+$, and by~\eqref{j97} and~\eqref{j98},  $v_{h-2-\alpha}\notin P[v_{i_1+1},v_{i_2-1}]$.
But this contradicts~\eqref{j94}.

\medskip
{\bf Case 2:} $d_{P}(v_p)= k-s$ and $p\geq 2k+2$. Then $v_p\in B'$ or $k$ is even.
%% $k-s=s=k/2$ and $v_p\in A$.
%If the latter holds, then because all $s=k-s$ neighbors of $v_p$ in $B$ are in $P$ (by (P1)), $v_{p-1}\in B$. But then $v_1v_{p-1}\in E(G)$ and $G$ has a cycle of length $p-1\geq 2k+1$, a contradiction.
%Thus we may assume  $v_p\in B'$.
 Lemma~\ref{cycle} together with~\eqref{j92} implies
\begin{equation}\label{j99}
d_P(v_1)=s \; \mbox{ and $P$ has crossing neighbors $v_{i_1}$ and $v_{i_2}$.}
\end{equation}

As in Case 1, let $C=P[v_1,v_{i_1}]\cup v_{i_1}v_p\cup P[v_p,v_{i_2}]\cup v_{i_2}v_1$. By the choice of $G$, $|C|\leq 2k-2$.
By~\eqref{j92'} and the definition of $C$, $B\subset V(C)$ and only one vertex  in $N_P(v_p)^+$ is not in $C$.
So since $B\cap N_P(v_p)^+=\emptyset$ and $|B|+ |N_P(v_p)^+|=s+(k-s)=k$,
\begin{equation}\label{j10}
|C|=2k-2\; \mbox{and each $v_i\in C\cap Y$ is in $B\cup N_P(v_p)^+$.}
\end{equation}
By~\eqref{j92'} and~\eqref{j10}, $N_P(v_1)=B$. In particular, $v_2\in B$. Repeating the proof of~\eqref{j95},
% (with $k+2$ in place of $k+3$),
 we derive that it holds also in our case.

Again, as in Case 1.1, if $v_1$ has no neighbors outside of $P$, we obtain $m^* \leq \frac{k}{2r-k+2}(n-1+D(G, X^*))$. So we may assume there exists $z\in N(v_1)-V(P)$ and a path $Q'=v_1,w_1,\ldots,w_j$  from $v_1$ through $w_1=z$ to $P-v_1$ internally disjoint from $P$. 
 Suppose $w_j=v_h$. Then repeating the proofs word by word, we derive that~\eqref{j96},~\eqref{j97} and~\eqref{j98} hold in
our case, as well.

Let $\beta\in \{0,1\}$ be such that $h-1-\beta$ is even.
Then by~\eqref{j98}, $v_{h-1-\beta}\notin B$, 
by~\eqref{j96}, $v_{h-1-\beta}\notin N_P(v_p)^+$, and by~\eqref{j97} and~\eqref{j98},  $v_{h-1-\beta}\notin P[v_{i_1+1},v_{i_2-1}]$.
But this contradicts~\eqref{j10}.

\medskip
{\bf Case 3:} $d_{P}(v_p)= k-s$ and $p= 2k$. We claim that
\begin{equation}\label{j101}
A'=\emptyset\; \mbox{and  $d_{G_0}(b')=k-s$ for each $b'\in B'$.}
\end{equation}
Indeed, if there is $a'\in A'$, then $a'$ has a nonneighbor $b\in B$ and since $G$ is saturated, it contains 
an $(a',b)$-path $P'$ with at least $2k$ vertices. Choose $P'$ to satisfy (P1), (P2), and (P3). By the case, $P'$ has exactly $2k$ vertices, and $A \subseteq V(P')$ otherwise we could extend $P'$. But then $|V(P')| \geq 2|A| + 1 > 2k$, a contradiction.
Similarly, if there is $b'\in B'$ with $d_{G_0}(b')\geq k-s+1$, then $b'$ has a nonneighbor $a\in A$, but any $(a,b')$-path $P'$ with at least $2k$
vertices contradicts the choice of $P$ in our case. This proves~\eqref{j101}.

If $|B'|\leq s$, then by~\eqref{j101}, instead of~\eqref{j102} we have
$$rm^* -D(G, X^*)\leq st+(m^*-t)(k-s)+(n-s)(k-s-1)+s=(2s-k)t+m^*(k-s)+n(k-s-1)-s(k-s-2),$$
which is maximized when $t= m^*$. Because $k-s-2\geq 0$, this yields
$$rm^*\leq ((2s-k)+(k-s))m^*+n(k-s-1) + D(G, x^*)=sm^*+n(k-s-1)+ D(G, X^*),$$
and hence as before, \[m^*\leq\frac{(k-s-1)n + D(G, X^*)}{r-s} \leq \frac{k}{2r-k+2}(n-1 + D(G, X^*)),\] a contradiction.  So suppose $|B'|\geq s+1$.

Recall that $d_A(v) = k-s$ for each $v\in B'$. By   Lemma~\ref{pcover}, either \hspace{1mm} (i)  there exists a set $A_1 \subset A$ with $|A_1|=k-s$  such that  $N_A(v) = A_1$
for each $v \in B'$, or (ii)~there exists  $B''\subset B'$ with $|B''| = k-s$ such that there exists a set of disjoint paths from $A$ to $A$ that covers $B''$.
If (ii) holds,
 then because $G[A \cup B]$ is complete bipartite, we can extend the set of paths to a cycle containing $B \cup B''$. This cycle must have at least $2k$ vertices, a contradiction.

Therefore we may assume that for each $(k-s)$-subset $B''$ of $B'$, we have $B'' \cup (A \cap N(B'')) = K_{k-s,k-s}$.
 Fix any $(k-s)$-subset $B'' \subset B'$. Let $G_2=G[A\cup B\cup B'']$. Since $G_2$ is the union of $K_{s,t}$ and $K_{k-s,k-s}$ with the intersection $A_1$, it has
 the following property: for each $a^*\in A-A_1$, $a_1\in A-a^*$ and  $b\in B\cup B''$,
 \begin{equation}\label{j141}
 \mbox{ $G_2$ has an $(a^,*a_1)$-path with $2k-1$ vertices and  an $(a^*,b)$-path with $2k-2$ vertices.}
\end{equation}
 
  Let $a^* \in A - A_1$. Since $a^* \notin A_1$,  $N_{G_2}(a^*) = B$. If also $N_G(a^*)= B$, i.e., $a^*$ has no neighbors outside of $P$, then again as in Case 1.1, we obtain $|X^*|\leq \frac{k}{2r-k+2}(n-1+D(G, X^*))$. So we may assume $a^*$ has a neighbor, say $y \in Y$, outside of $B \cup B'$. Let $Q = a^*, y, w_1, \ldots, w_\ell$ be a path internally disjoint from $A\cup B\cup B'$ such that $w_\ell \in A\cup B\cup B'$. Such a path exists because $G$ is 2-connected. Note that if  $w_{\ell}\in B\cup B''$, then by parity, $\ell\geq 2$. Then $Q$
  together with a path in $G_2$ satisfying~\eqref{j141} forms a cycle of length at least $2k$. 
\qed

\section{Proof of Theorem~\ref{main} for 2-connected graphs}
Recall the statement of the Main Theorem for bipartite graphs.

{\bf Theorem~\ref{main}}. {\em Let $ k\geq 4$, $r\geq k+1$  and $m,m^*,n$ be positive integers with $n\geq k$, $m \geq m^*\geq  k-1$ and $m\geq k$. 
Let $G=(X,Y;E)$ be a bipartite $2$-connected  graph with parts $X$ and $Y$, where $|X|=m$, $|Y|=n$, and
let $X^*\subseteq X$ with $|X^*|=m^*$.  If $c(G)<2k$, then  \begin{equation}\label{e2}
 m^*\leq  \frac{k}{2r-k+2} (n-1 + D(G, X^*)).
\end{equation}
%or $r = k$, $n \leq k-1$, and $m^* \leq n-k+1 + D(G, X^*)$.}

\begin{proof}

%Every $2$-connected graph has a cycle. So the theorem is trivially true for $k=2$. 

 Let $G = (X,Y;E)$ and $X^* \subseteq X$ be an edge-maximal counterexample with minimum $|X|$. Note that adding edges to $G$ can only decrease the deficiency while $m^*, m,$ and $n$ stay the same. So we may assume that $G$ is $2k$-saturated, i.e., adding any additional edge 
connecting $X$ with $Y$ creates a cycle of length at least $2k$. Therefore, 
\begin{equation}\label{e5}
\mbox{ \em for any nonadjacent 
$x \in X$ and $ y \in Y $, there is an $(x,y)$-path on at least $2k$ vertices.}
\end{equation}

If $n \leq k-1$ then each vertex $x \in X$ has $D_G(x) \geq r - n \geq r-k+1$. Then for $r \geq k+1$, we get
\[\frac{k}{2r-k+2}(n-1+D(G, X^*)) \geq \frac{k}{2r-k+2}(n-1+m^*(r-k+1)) \geq \frac{k}{2r-k+2}m^*(r-k+1) \geq m^*,\] 

where the last inequality holds whenever $k(r-k+1) \geq 2r-k+2$, i.e., whenever $r \geq k + \frac{2}{k-2}$. 

%For the case $n \leq k-1$ and $r=k$, because $D(G, X^*) \geq (k-n)m^*$, we have $m^* \leq n-(k-1) + D(G, X^*)$, as desired.

Thus we may assume from now on that $n \geq k$. 

 Our first claim is:
\begin{equation}\label{e3}
e(G)>\left\lfloor\frac{k-1}{2}\right\rfloor m^*+\left\lceil\frac{k-1}{2}\right\rceil (n-1).
\end{equation}

Indeed, $e(G)\geq rm^* - D(G, X^*)$. So, if~\eqref{e3} fails and $k$ is odd, then $rm^* - D(G,X^*)\leq \frac{k-1}{2}(m^*+(n-1))$. Solving for $m^*$, we get
$$m^*\leq \frac{(k-1)(n-1+D(G, X^*))}{2r-k+1}.$$
Since $r\geq k$, this yields~\eqref{e2}, a contradiction to the choice of $G$. So suppose $k$ is even. Then
$rm^* + D(G, X^*)\leq \frac{k}{2}(m^*+(n-1))-m^*$. Solving for $m^*$ and using $k\geq 4$ and $r\geq k$, we get
$$m^*\leq\frac{k(n-1 + D(G, X^*))}{2r-k+2},$$
 and the theorem holds.
This proves~\eqref{e3}.

\medskip
Apply a mixed $(\left\lfloor\frac{k-1}{2}\right\rfloor,\left\lceil\frac{k-1}{2}\right\rceil)(X^*, Y)$-disintegration to $G$, that is, first delete all vertices in $X- X^*$ and then consecutively delete
vertices of degree at most $\left\lfloor\frac{k-1}{2}\right\rfloor$ in $X$ and vertices of degree at most $\left\lceil\frac{k-1}{2}\right\rceil$ in $Y$.
Let $G'$ be the resulting graph 
with parts $A\subseteq X^*$ and $B\subseteq Y$. 
Suppose first that $G'$ is empty. Then at each step of the disintegration process, we lose at most $\left\lfloor\frac{k-1}{2}\right\rfloor$ edges if a vertex in $X^*$ is deleted and at most $\left\lceil\frac{k-1}{2}\right\rceil$ edges if a vertex in $Y$ is deleted. Furthermore, when we arrive to the last $\left\lfloor\frac{k-1}{2}\right\rfloor + \left\lceil\frac{k-1}{2}\right\rceil = k-1$ vertices in the disintegration process, there exists at most $\left\lfloor\frac{k-1}{2}\right\rfloor \cdot \left\lceil\frac{k-1}{2}\right\rceil$ edges. Thus
\begin{eqnarray*}
e(G) &\leq & \left\lfloor\frac{k-1}{2}\right\rfloor m^*+\left\lceil\frac{k-1}{2}\right\rceil n - 
\left\lfloor\frac{k-1}{2}\right\rfloor\left(k-1\right) + 
\left\lfloor\frac{k-1}{2}\right\rfloor \cdot \left\lceil\frac{k-1}{2}\right\rceil \\
& = & \left\lfloor\frac{k-1}{2}\right\rfloor m^*+\left\lceil\frac{k-1}{2}\right\rceil n - 
\left\lfloor\frac{k-1}{2}\right\rfloor^2 
 \leq  \left\lfloor\frac{k-1}{2}\right\rfloor m^*+\left\lceil\frac{k-1}{2}\right\rceil (n-1),
\end{eqnarray*}
contradicting~\eqref{e3}.
% then $|E(G)|\leq \left\lfloor\frac{k-1}{2}\right\rfloor m+\left\lceil\frac{k-1}{2}\right\rceil n$,
% contradicting~\eqref{e3}.
Therefore $G'$ is not empty,
\begin{equation}\label{j142}
\mbox{ \em $d_{G'}(a)\geq 1+ \left\lfloor\frac{k-1}{2}\right\rfloor$ for each $a\in A$, and  $d_{G'}(b)\geq 1+\left\lceil\frac{k-1}{2}\right\rceil$ for each $b\in B$. }
\end{equation}

\medskip
{\bf Case 1:} $G'$ is a complete bipartite graph .
Let $s=\min\{|A|,|B|\}$ and $t=\max\{|A|,|B|\}$. Since  $c(G)\leq 2k-2$, by~\eqref{j142},
\begin{equation}\label{j15} 
\mbox{ $\frac{k}{2}\leq s\leq  k-1$, and  if $s=|A|$, then $s\geq\frac{k+1}{2}$. }
\end{equation}

Moreover, suppose $s=k-1$. Then $G$ contains a $K_{k-1, k-1}$ with parts $A$ and $B'$ where $B' \subseteq B$. Let $u \in G - (A \cup B')$. Such a vertex exists because $m, n \geq k$.  Because $G$ is 2-connected, there exists two internally disjoint paths $P_1$ and $P_2$ from $u$ to $A \cup B'$ such that $P_1$ has endpoints $u$ and $u_1$ and $P_2$ has endpoints $u$ and $u_2$, and these paths only interesect $A \cup B'$ at $u_1$ and $u_2$ respectively. If $|V(P_1 \cup P_2)| \geq 4$, that is, $P_1 \cup P_2$ contains a vertex in $G - (A \cup B')$ other than $u$, then we may find a path $P_3$ in $ A \cup B'$ of length $2k-2$ if $u_1$ and $u_2$ are in different partite sets, or of length $2k-3$ if they are in the same partite set. Then $P_1 \cup P_2 \cup P_3$ yields a cycle of length at least $2k$. Therefore $P_1 \cup P_2 = u_1, u, u_2$. Next let $w$ be a vertex in $G - (A\cup B')$ in the opposite partite set than that of $u$ (again such a vertex exists because $n, m \geq k$). Similarly, all internally disjoint paths $Q_1, Q_2$ connecting $w$ to $A \cup B'$ must be of the form $Q_1 \cup Q_2 = w_1 w w_2$ for some $w_1, w_2 \in A \cup B'$. Thus we may find disjoint paths $R_1$ and $R_2$ partitioning $V(A \cup B')$ such that $R_1$ has endpoints $u_1$ and $w_1$ and $R_2$ has endpoints $u_2$ and $w_2$. Then $P_1 \cup R_2 \cup Q_2 \cup R_1$ yields a cycle of length $2k$, a contradiction. Therefore
%\begin{equation}\label{j91'}
%\mbox{$s\leq k-2$ (which yields $k\geq 4$, since $s\geq k/2$).}
%\end{equation}
$ s \leq k-2$.

{\bf Case 1.1:} $s=|B|$. 
%By~\eqref{e3} and the definition of $G'$, $st>\frac{k-2}{2}t+\frac{k}{2}s$.
%Solving for $t$ and using~\eqref{j15}, we have
%$$t>\frac{ks}{2s-k+2}=\frac{k}{2-\frac{k-2}{s}}\geq \frac{k}{2-\frac{k-2}{k-1}}=k-1.
%$$
%Thus $t\geq k$, and Lemma~\ref{nl2} yields the theorem in this case.

For $k$ odd, by~\eqref{e3} and the definition of $G'$, $st > \frac{k-1}{2}(t + s - 1)$. Solving for $t$ and using~\eqref{j15}, we have
\[t > \frac{(k-1)(s-1)}{2s-k+1} = \frac{k-1}{\frac{2(s-1)}{s-1} + \frac{2}{s-1} - \frac{k-1}{s-1}} = \frac{k-1}{2 - \frac{k-3}{s-1}} \geq \frac{k-1}{2 - \frac{k-3}{(k-2)-1}} = (k-1).\]

For $k$ even, we instead get $st > \frac{k-2}{2}t + \frac{k}{2}(s-1)$ and so
\[t > \frac{k(s-1)}{2s-k+2} = \frac{k}{2 - \frac{k-4}{s-1}},\]
so $t \geq k$ except in the case where $s = k-2$ and $t=k-1$. 
So suppose $|B| = k-2$ and $|A| = k-1$.

By Lemma~\ref{c1}, either $|N(Y - B)| = 2$ or $|N(Y - B) \cap A| \leq 1$. Suppose the first case holds. Let $N(Y-B) = \{x_1, x_2\}$ so that each vertex in $X^* - \{x_1, x_2\}$ has neighbors only in $B$. Without loss of generality, first assume that $x_1 \in X^*$. Then \[D(G, X^*) \geq D_G(x_1) + (|X^*|- 2)(r-k+2).\]  Also $n-1 + D_G(x_1) \geq r-1$. Thus using the fact that $|X^*| \geq k-1$,  we have
%\begin{eqnarray*}\label{y-b}
$$
\frac{k(n-1+D(G, X^*))}{2r-k+2}  - |X^*| \geq \frac{k(r-1 + (|X^*| - 2)(r-k+2))}{2r-k+2} - |X^*|$$
$$
=  \frac{k((|X^*| - 1)(r-k+2)+k-3)}{2r-k+2} - |X^*|
\geq  \frac{k((k-2)(r-k+2)+k-3)}{2r-k+2} - (k-1)
\geq  0.$$
%\end{eqnarray*}
where the last inequality holds whenever $r \geq k + \frac{2}{k(k-4)+2} - 2$. Therefore \[|X^*| \leq \frac{k}{2r-k+2}(n-1+D(G, X^*)),\] a contradiction. The case where $X^*$ contains neither $x_1$ nor $x_2$ is similar (and easier) as we would have $D(G, X^*) = |X^*|(r-k+2)$.

So we may assume that $|N(Y-B) \cap A| \leq 1$ but $|N(Y-B)| \neq 2$. If $|X^*| = k-1$, i.e., $X^* = A$, then all vertices in $X^*$ but at most one have neighbors only in $B$. Then just as in the previous case, we have \[|X^*| \leq \frac{k}{2r-k+2}(n-1+D(G, X^*)).\]
Also, if $|X| = k$, then there is a single vertex $x' \in X - A$. Because $|N(Y-B) \cap A|\leq 1$ and $G$ is 2-connected, all vertices in $Y - B$ must also be adjacent to $x'$. But then $|N(Y-B)| = 2$, a contradiction.

So we may assume that $|X^*| \geq k$ and $|X| \geq k+1$. Fix $x \in A - N(Y-B)$. Then $x$ is contained in a $K_{k-1, k-2}$ subgraph of $G$ and has no neighbors outside of this subgraph. It is easy to see then that $G - x$ is 2-connected. Furthermore, $|X^* - x| \geq k-1$, $|X - x| \geq k$ and $D(G- x, X^* -x ) = D(G, X^*) - (r-k+2)$, contradicting Lemma~\ref{2connind}. This completes the proof that $t=|A| \geq k$. Applying Lemma~\ref{nl2} completes the case. 

\medskip
{\bf Case 1.2:} $s=|A|$. By~\eqref{e3}, $s\geq\frac{k+1}{2}$.
Apply the $(k-s)(X^*, Y)$-disintegration to $G$, and let $G''$ be the resulting graph. If $G'' \neq G'$, then there is some $u\in V(G'')-V(G')$
not adjacent to some $v\in V(G')$ in the other partite set. 
By~\eqref{e5}, $G$ contains a $(u,v)$-path  $P'$ on at least $2k$ vertices.
Choose a path $P''$ of maximum length in $G$ whose both endpoints are in $G''$ and at least one of them in $G'$. 
Let $P''=v_1,\ldots, v_p$ where  $v_1\in V(G')$. In view of $P'$,  $p\geq 2k$.
By the maximality of $P''$, all neighbors of $v_p$  in $G''$ and all neighbors of $v_1$ in $G'$  lie in $P''$. 
So, $d_{P''}(v_1)+d_{P''}(v_p)\geq (k-s+1)+s=k+1$. 
 By Lemma~\ref{cycle}, $v_1$ and $v_p$ are in the same partite set (i.e. $p$ is odd) and have crossing neighbors in $P''$. 
 So, $P''$ satisfies the conditions of Theorem~\ref{cycle1}
 for $\alpha=k-s$, therefore $P''$ is in crossing formation. But this contradicts the Main Lemma.  
Thus $G' = G''$, i.e., everything except for $G'$ is removed in the weaker $(k-s)(X^*,Y)$-disintegration.

Next, we apply a mixed disintegration process to $G - (A \cup B)$ where vertices in $Y-B$ are removed (iteratively) if at the time of deletion they have at most $k-s-1$ neighbors  within $X^*- A$, and vertices in $X* - A$ are removed if they have at most $k-s$ neighbors total. Let $G'''$ be the resulting graph. We claim that also \begin{eqnarray}\label{g'''}G''' = G'.\end{eqnarray} Suppose not. Then there exists a non-edge between $A\cup B$ and $G''' - (A \cup B)$. Among such nonadjacent vertices, choose a pair $v_1 \in A \cup B$, $v_p \in G''' -( A \cup B)$ such that a path $P = v_1, \ldots, v_p$ between them is longest possible.  Thus all neighbors of $v_1$ in $A \cup B$ are in $P$, and all neighbors of $v_p$ in $G''' - (A \cup B)$ are in $P$. 

First observe that if $v_1 \in A$, then by the maximality of $P$, all vertices in $B$ (which are neighbors of $v_1$) appear in $P$. Thus by Lemma~\ref{cycle}, we may find a cycle that contains all of $B$. Such a cycle contains at least $2|B|\geq 2k$ vertices, a contradiction. So we assume $v_1 \in B$. If $v_p \in X- A$, then because $d_P(v_p) \geq k-s+1$, Lemma~\ref{cycle} implies that $G$ contains a cycle of length at least $2(s + (k-s+1) - 1) = 2k$. Therefore $v_p \in Y - B$, and $v_p$ has at least $k-s$ neighbors from $X - A$ in $P$. Since $v_1$ has $s$ neighbors in $A$, we can find a cycle that covers $N_P(v_1) \cup N_P(v_p)$. Note that $N_P(v_1)$ and $N_P(v_p)$ are the same parity. Thus such a cycle has at least $2(k-s + s)$ vertices, a contradiction. Thus proves~\eqref{g'''}.

{\em Case 1.2.1:} $t \geq r$.
For simplicity, let $D' = D(G, X^* - A)$  denote the deficiency of vertices in $X^* - A$. 

For any $v$ be any vertex in $X^* - A$, because $v$ was deleted in the first disintegration, $v$ has at most $k-s$ neighbors in $B$. Thus $v$ has at most $(n-t) + (k-s)$ neighbors. Since $d(v) + D(v) \geq r$, $(n-t) + (k-s) + D(v) \geq r$ which implies that 

\begin{equation}\label{n-t+D}
n-t+D' \geq r-k+s \geq r- \frac{k-1}{2}
\end{equation}

By~\eqref{g'''}, we obtain that $r(m^*-s) +D' \leq (k-s)(m^*-s) + (k-s-1)(n-t)$. Solving for $m^*$, we get \[m^*\leq \frac{k-s-1}{r-k+s}(n-t) + D' + s \leq \frac{k-s-1}{r-k+s}(n-t+D') + s.\]
 We first show that for fixed $r, k, n, t, D'$, the function $f(s):= \frac{k-s-1}{r-k+s}(n-t+D') + s$ is decreasing in $s$. Indeed, taking the first derivative, we have

\[f'(s)  = \frac{-(r-k+s) - (k-s-1)}{(r-k+s)^2} (n-t+D') + 1 = \frac{-(r-1)(n-t+D')}{(r-k+s)^2} + 1.\]
Since $r-1 > r-k+s$ and $n-t + D'> r - k + s$, $ \frac{-(r-1)(n-t)}{(r-k+s)^2} < -1$, therefore it is maximized at $s = \frac{k+1}{2}$.

\begin{eqnarray*}
 m^* &\leq & \frac{(k-\frac{k+1}{2}-1)(n-t) - D'}{r-k+\frac{k+1}{2}} + \frac{k+1}{2} 
 <  \frac{(k-3)(n-t + D')}{2r-k+1} + \frac{k+1}{2}\\
 & = & \frac{(k-1)(n-t+D')}{2r-k+1} - \frac{2(n-t+D')}{2r-k+1} + \frac{k+1}{2}\\
 & \leq & \frac{(k-1)(n-1+D')}{2r-k+1} - \frac{(k-1)(t-1)}{2r-k+1} - \frac{2(r-\frac{k-1}{2})}{2r-k+1}  + \frac{k+1}{2}\\
 & \leq & \frac{(k-1)(n-1+D')}{2r-k+1} - \frac{(k-1)(r-1)}{2r-k+1} + (\frac{k+1}{2} - 1) \\
 & \leq & \frac{(k-1)(n-1+D')}{2r-k+1} - \frac{(k-1)(r-1)}{2r-k+1} + \frac{(k-1)(r-\frac{k-1}{2})}{2r-k+1}
  <  \frac{(k-1)(n-1+D')}{2r-k+1},
\end{eqnarray*}
which is less than $\frac{k}{2r-k+1}(n-1+D(G, X^*))$. 

{\em Case 1.2.2}: $t \leq r$.
For simplicity, let $D = D(G, X^*)$.
We have that $rm^*- D \leq e(G) \leq st + (k-s)(n-t) + (k-s)(m^*-s)$. Solving for $m^*$, we have
\[m^* \leq \frac{s(t-k+s) + (k-s)(n-t) + D}{r-k+s} \leq \frac{s(t-k+s) + (k-s)(n-t) + D}{r-k+s}.\]
Again, it can be shown that this function is decreasing with respect to $s$, and so it is maximized when $s = \frac{k+1}{2}$. Furthermore,  the function is maximized whenever $t$ is as large as possible, i.e., when $ t = r$.  
Therefore

\begin{eqnarray*}
m^* & \leq & \frac{(k-1)(n-r+ D)}{2r-k+1} + \frac{k+1}{2} \\
& = & \left[\frac{k(n-r+D)}{2r-k+2} - \frac{(2r-2k+2)(n-r+D)}{(2r-k+2)(2r-k+1)} \right] + \frac{k+1}{2}\\
& = & \frac{k(n-1+D)}{2r-k+2} - \frac{k(r-1)}{2r-k+2} - \frac{(2r-2k+2)(n-r+D)}{(2r-k+2)(2r-k+1)} + \frac{k+1}{2}\\
&\leq & \frac{k(n-1+D)}{2r-k+2} - \frac{k(r-1)}{2r-k+2} - \frac{(2r-k+2 - k) (r-\frac{k-1}{2})}{(2r-k+2)(2r-k+1)} + \frac{k+1}{2} \\
&= & \frac{k(n-1+D)}{2r-k+2} - \frac{k(r-1)}{2r-k+2} - \frac{1}{2} \left(1-\frac{k}{2r-k+2}\right) + \frac{k+1}{2} \\
&= & \frac{k(n-1+D)}{2r-k+2} - \frac{k(r-1 - \frac{1}{2})}{2r-k+2} + \frac{k}{2} \\
&= & \frac{k(n-1+D)}{2r-k+2} - \frac{k(r-1 - \frac{1}{2})}{2r-k+2} + \frac{k(r-\frac{k-2}{2})}{2r-k+2} 
\leq \frac{k(n-1+D)}{2r-k+2},
\end{eqnarray*}

a contradiction.\footnote{Note that the last inequality holds whenever $k \geq 5$. If $k = 4$, then instead of  $s = (k+1)/2$ we substitute $s = k/2$ and obtain the same inequality in the end.}
\medskip

{\bf Case 2:} $G'$ is not a complete bipartite graph.
  Let $P=u_1,\ldots,u_q$ be a longest path in $G$ whose both ends are in $V(G')$, and subject to this, $\sum_{i=1}^q d_P(u_i)$ is maximized. 
By~\eqref{e5} and the case, $q\geq 2k$.
 By the maximality of $P$, all neighbors of $u_1$ and of $u_q$ in $G'$  lie in $P$.
 
\medskip
 {\bf Case 2.1:} $k$ is odd.
  Then $P$ satisfies the conditions of Theorem~\ref{cycle1}
 for $\alpha=\frac{k-1}{2}$.
%So, $d_P(u_1)+d_P(u_q)\geq k+1$. 
But then $P$ is in saturated crossing formation, contradicting the Main Lemma.

\medskip
{\bf Case 2.2:} $k$ is even. By~\eqref{j142},  $d_{G'}(a)\geq \frac{k}{2}$ for each $a\in A$ and  $d_{G'}(b)\geq\frac{k+2}{2}$ for each $b\in B$.
Since $c(G)\leq 2k-2$, by Lemma~\ref{cycle}, 
\begin{equation}\label{a271}
\mbox{ \em $q$ is odd, $\{u_1,u_q\}\subseteq A$,   $d_{G'}(u_1)=d_P(u_1) \in \{\frac{k}{2}, \frac{k}{2} + 1\}, d_P(u_q)=d_{G'}(u_q) \in \{\frac{k}{2}, \frac{k}{2} + 1\} $.}
\end{equation}

First we show the following claim.
\begin{claim} Path $P$ has a pair of crossing neighbors.\end{claim}
\begin{proof}
Suppose not. %By Lemma~\ref{cycle}, $d_{G'}(u_1) = d_{G'}(u_q) = k/2$, otherwise $c(G) \geq 2k$. 
 Suppose the largest index of a neighbor of $u_1$ in $P$ is $j_1$ and 
the smallest index of a neighbor of $u_q$ in $P$ is $j_2$. Since $P$ has no crossing neighbors, $j_1\leq j_2$.
If $j_1<j_2$ or  $d_{G'}(u_1) + d_{G'}(u_q) \geq k+1$, then by the "furthermore" part of Lemma~\ref{cycle}, $G$ has a cycle 
with at least
$ k$ vertices in $B$, a contradiction to $c(G)\leq 2k-2$. Thus $d_{G'}(u_1) = d_{G'}(u_q) = k/2$ and $j_1=j_2$.

By the definition of $P$, $q\geq 2k+1$. By symmetry, we may assume $j_1\geq k+1$. Since $k$ and $j_1$ are even, this
yields
\begin{equation}\label{j27}
j_1\geq k+2\quad \mbox{\em and hence $u_1$ has a nonneighbor $u_j$ for some even $j<j_1$.}
\end{equation}
Since $G$ is $2$-connected, $G-u_{j_1}$ has a path $P_1$ that is internally disjoint from $P$ connecting $P[u_1,u_{j_1-1}]$ with $P[u_{j_1+1},u_q]$.
Among such paths, choose a path $P_1=w_1,\ldots, w_{\ell}$ with $w_1=u_{j_3}$ and $w_\ell=u_{j_4}$ so that $j_3<j_1<j_4$ and
$j_3$ is as small as possible. Let $j_5$ be the smallest index such that $j_5>j_3$ and $u_{j_5}u_1\in E(G)$ and let
 $j_6$ be the largest index such that $j_6<j_4$ and $u_{j_6}u_q\in E(G)$
Since $j_3<j_1=j_2<j_4$, indices
$j_5$ and $j_6$ are well defined and $j_5\leq j_1\leq j_6$.
If $j_3=1$, then by the definition of $j_1$, $\ell\geq 3$ and hence $w_2\in Y-V(P)$. Thus the cycle
$P_1\cup P[u_{j_4},u_q]\cup u_q u_{j_6}\cup P[u_1,u_{j_6}]$ has at least $k$ vertices in $Y$: $N_P(u_1)\cup N_P(u_q)\cup \{w_2\}$. This contradicts
to $c(G)\leq 2k-2$. Therefore
\begin{equation}\label{a27}
j_3\geq 2.
\end{equation}
Cycle $C_1=P[u_1,u_{j_3}]\cup P_1\cup P[u_{j_4},u_q]\cup u_qu_{j_6}\cup P[u_{j_5},u_{j_6}]\cup u_{j_5}u_1$ contains $N_P(u_1)\cup N_P(u_q)\cup V(P_1)$. Since $c(G)\leq 2k-2$, $V(C_1)\cap Y=N_P(u_1)\cup N_P(u_q)$.
In particular, $|V(C_1)| = 2|k/2 + k/2 - 1| = 2k-2$, and
\begin{equation}\label{a272}
\mbox{ \em for every even $2\leq j\leq j_3$ and $j_5\leq j\leq j_1$,  $u_1u_j\in E(G)$,}
\end{equation}
and similarly for $u_q$, 
\begin{equation}\label{a272'}
\mbox{ \em for every even $j_1\leq j\leq j_6$ and $j_4\leq j\leq q-1$,  $u_qu_j\in E(G)$.}
\end{equation}
From~(\ref{j27}) and~(\ref{a272}) we conclude
\begin{equation}\label{j271}
u_1u_{j_5-2}\notin E(G).
\end{equation}
Now we will show that
\begin{equation}\label{u1}
\mbox{ \em $u_1$ has a neighbor outside of $P$.}
\end{equation}

Suppose not. Since $N_P(u_1) \subseteq G'$, we have $u_2 \in G'$. In particular, $u_2$ has at least $\frac{k}{2} + 1$ neighbors in $G'$. 

If $u_2$ has a neighbor $u_r$ in $P$ that is a successor of some neighbor, say $u_s$ of $u_q$, that is, $r = s+1$, then the cycle $u_2 u_{s+1}  \cup P[u_{s+1}, u_q] \cup u_q u_s \cup P[u_s, u_1]$ has length $|V(P)|-1 \geq 2k$, a contradiction. 

Suppose that $u_2$ is adjacent to a vertex in $u_r$ with $j_6 < r <j_4$. 
%Without loss of generality, we could have picked $u_1$ and $u_q$ so that $j_5 - j_3 \geq j_4 - j_6$ (otherwise, just let $P = u_q \ldots u_1$). If $j_4 - j_6 = 2$, $u_r = u_{j_6 + 1}$, a successor of a neighbor of $u_q$, a contradiction. So we may assume $j_4 - j_6 \geq 3$ and therefore $j_5 - j_3 \geq 3$. Note that this implies $u_{j_5 - 2} \notin N(u_1)$.
Then the cycle $C' = u_2u_r \cup P[u_r, u_q] \cup u_q u_{j_6} \cup P[u_{j_6},u_2]$  %includes 
%all vertices in $P$ except for $\{u_1\} \cup \{u_{j_6+1}, \ldots u_{r-1}\}$ where $r-1 < j_4 - 1$. In particular $C'$ 
 contains  $N_P(u_1) \cup N_P(u_q) \cup \{u_{j_5 - 2}\}$. By~(\ref{j27}) this means $C'$ has at least $k$ vertices in $Y$.
 Thus $c(G) \geq 2k$, a contradiction.

Therefore, by~\eqref{a272'}, if $u_2$ has neighbors in $P$, they  appear in $P[u_1, \ldots, u_{j_1-1}]$.
%(and by symmetry, all neighbors of $u_{q-1}$ in $P$ appear in $P[u_{j_2+1}, u_q]$).
 If $u_2$ has a neighbor $u_r$ with $j_3 < r < j_5$, then the cycle $P[u_2, u_{j_3}] \cup P_1 \cup P[u_{j_4} u_q] \cup u_q u_{j_6} \cup P[u_{j_6}, u_{r}] \cup u_r u_2$ is longer than $C_1$ except when $u_r = u_{j_5-1}$. This implies that 

\begin{equation}\label{u2}\mbox{each neighbor of $u_2$ in $P$ is a predecessor of a neighbor of $u_1$.}
 \end{equation}
 Similarly,
\begin{equation}\label{uq-1}\mbox{each neighbor of $u_{q-1}$ in $P$ is a successor of a neighbor uf $u_q$. }
\end{equation}

Since $d_{G'}(u_2) \geq k/2 + 1$, and $d_G(u_1) = k/2$, $u_2$ has at least 1 neighbor outside of $P$ in $G'$. Call this neighbor $u_1'$. By definition of $G'$, $d_G(u_1') \geq k/2$. Let $P' = u_1' \cup P[u_2, u_q]$, where $|V(P')| = |V(P)|$. If $u_1'$ has a neighbor $v \in V(G') - V(P')$, then the path $v\cup P'$ is a longer path with endpoints in $G'$, contradicting the choice of $P$. Therefore, by Lemma~\ref{cycle}, $d_{G'}(u_1') = d_{P'}(u_1') = d_P(u_1) = k/2$, and we may assume $P'$ has no crossing neighbors.  If $u_1'$ has any neighbors in $G$ outside of $P$, then we instead consider the path $P'$ and vertex $u_1'$ and arrive at~\eqref{u1}.

Define the indices $j_1', j_3'$, and $j_5'$ as before in view of $P'$ and $u_1'$. By symmetry, we have $j_1' = j_1$. Because $P' - u_1' = P - u_1$ and by~\eqref{a27}, $j_3' = j_3$. Finally, if $j_5' < j_5$, then $C_2=P[u_1',u_{j_3'}]\cup P_1\cup P[u_{j_4},u_q]\cup u_qu_{j_6}\cup P[u_{j_5'},u_{j_6}]\cup u_{j_5'}u_1'$ is a longer cycle than $C_1$. If $j_5' > j_5$, then $u_1'$ must have less than $k/2$ neighbors in $P'$, a contradiction. Therefore $j_5' = j_5$ and again by symmetry,~\eqref{a272} and~\eqref{a272'} hold for $P'$ and $u'_1$.

Thus $N_G(u_1) \subseteq N_G(u_1')$. 
%If $u_1'$ has neighbors outside of $P'$ then instead of $u_1$ and $P$, we consider $u_1'$ and $P'$, which proves~\eqref{u1}. So suppose $N_G(u_1') = N_P(u_1')$, i.e., $u_1'$ and $u_1$ have the same neighborhoods. We consider $G - u_1$ and show that it is 2-connected. If not, then there exists some cut vertex $v$ of $G- u_1$, and therefore $\{v, u_1\}$ must form a cut set. There must exist some $u, w \in V(G)$ such that each $(u, w)$-path intersects $v$ or $u_1$. if all such paths intersect $v$, then $v$ is a cut vertex of $G$, a contradiction. So let $Q$ be a $(u, w)$-path that intersects $u_1$ but not $v$. If $v \neq u_1'$, then we may replace $u_1$ with $u_1'$ in $Q$ to obtain a path that does not intersect $\{u_1, v\}$, a contradiction. Therefore $u_1' = v$. Since $u_1$ has neighbors only in $P$, both the predecessor and successor of $u_1$ in $Q$ are in $P$. That is, $Q$ intersects $P - u_1$ at least twice. Let $Q = q_1 \ldots q_r \ldots q_s \ldots q_t$ where $q_r$ is the first time $Q$ intersects $P - u_1$ and $q_s$ the last. Then $Q = a_q \ldots q_r \cup P[q_r, q_s] \cup q_s \ldots q_t$ is a $(u,w)$-path that avoids $\{u_1, u_1'\}$, a contradiction.
%Thus $G- u_1$ is 2-connected. Add edges until $G - u_1$ is $2k$-saturated, and let $G'$ be the resulting graph.
Note that $|X^* - u_1| \geq |N(u_2) \cup N(u_{q-1}) | - 1 \geq k/2+1 + k/2+1 - 1 = k+1$, where the last inequality holds because if $u_2$ and $u_{q-1}$ shared a neighbor $v$ (note that it cannot be in $P$ by~\eqref{u2} and~\eqref{uq-1}) then $u_2 vu_{q-1} \cup P[u_2, u_{q-1}]$ is a cycle with $|V(P)| - 1 \geq 2k$ vertices, a contradiction. Applying Lemma~\ref{ind} gives a contradiction, hence~\eqref{u1} holds. 

Let $z_1 \in N_G(u_1) - V(P)$. 
As $G$ is $2$-connected, $G-u_1$ contains a path $P_2=z_1,\ldots, z_m$ from $z_1$ to $V(P\cup P_1)-u_1$.
By~\eqref{a27}, $z_m\in V(P)$, say $z_m=u_{j_7}$. Again by~\eqref{a27}, $j_7\leq j_1$.

If $j_3<j_7\leq j_5$, then the cycle $C_2=P[u_1,u_{j_3}]\cup P_1\cup P[u_{j_4},u_q]\cup u_qu_{j_6}\cup P[u_{j_7},u_{j_6}]\cup 
P_2\cup z_1u_1$ contains at least $k$ vertices from $Y$, namely, $N_P(u_1)\cup N_P(u_q)\cup \{z_1\}$, a contradiction.
Suppose now that $j_7\in \{2i-1,2i\}$ for some $i\in \{2,\ldots,j_3/2\}\cup \{1+j_5/2,\ldots,j_1/2\}$. By~\eqref{a272},
$u_1u_{2i-2}\in E(G)$ and so by~\eqref{a271}, $u_{2i-2}\in V(G')$. But  the path
$P_3=P[u_{2i-2},u_1]\cup u_1z_1\cup P_2\cup P[u_{j_7},u_q]$ is longer than $P$, contradicting the choice of $P$.
Finally, if $j_7 = 2$, then we instead take the path $u_1 z_1 \cup P_2 \cup P[u_2, u_q]$. This proves the claim.
\end{proof}

Let $u_{i_1}$ and $u_{i_2}$ be the first occurring pair of crossing neighbors on $P$.

\begin{claim}\label{u1cycle}If $|X^*| \geq k+1$, then every cycle in $G$ containing $N(u_1)$  also contains $u_1$, and every cycle containing $N(u_q)$  also contains $u_q$. 
\end{claim}

\begin{proof}
We prove the claim  for $u_1$. The result for $u_{q}$ follows by symmetry. 
%Suppose $u_2 u_r \in E(G)$ where $r \in \{i_1 +1, \ldots, i_2 - 1\}$ and consider the cycle \[C = P[u_2, u_{i_1}] \cup u_{i_1}u_q \cup P[u_q, u_{r}] \cup u_{r}u_2.\] The cycle $C$ contains $N(u_1)$ and does not contain $u_1$. 
Suppose there exists a cycle $C$ that contains $N(u_1)$ but not $u_1$. If $G - u_1$ has a cut vertex $v$, then because $G$ is 2-connected, $\{v, u_1\}$ is a cut set of $G$. Therefore there exist  vertices $u_i, u_j \in N_G(u_1)$ that are in distinct components of $(G - u_1) - v$. Let $P'$ be a segment of $C$ from $u_i$ to $u_j$ not containing $v$. Then $P'$ is a path from $u_i$ to $u_j$ in $(G- u_1) - v$, a contradiction. Therefore $G - u_1$ is 2-connected, contradicting Lemma~\ref{2connind}.
\end{proof}

\begin{claim}Each of $u_1$ and $u_q$  has at least one neighbor outside of $P$.\end{claim}

\begin{proof}
Similarly to the proof of Case~1 of Lemma~\ref{nl2},
let $C=P[u_1,u_{i_1}]\cup u_{i_1}u_q\cup P[u_{i_2},u_q]\cup u_{i_2}u_1$. By the choice of $G$, $|C|\leq 2k-2$.
Since $N_P(u_1)^- \cap N_P(u_q)^+=\emptyset$, $|N_P(u_1)^-|+ |N_P(u_q)^+|\geq k$, and
$C$ does not contain only two vertices in $N_P(u_1)^-\cup N_P(u_q)^+$,
$$2k-2\geq |C|\geq 2k-4\; \mbox{{\em and}  $|(N_P(u_1)^-\cup N_P(u_q)^+)\cap   C\cap A|\geq d_P(u_1)+d_P(u_q)-2$ .}$$
This means
\begin{equation}\label{j152}
\parbox{14.5cm}{\em either $ |C|= 2k-4$ and each $u_i\in C\cap A$ is in $N_P(u_1)^-\cup N_P(u_q)^+$, or
 $ |C|= 2k-2$ and there is at most one $i_0$ such that $u_{i_0}\in (C\cap A)-(N_P(u_1)^-\cup N_P(u_q)^+)$.
}
\end{equation}

%Also by Lemma~\ref{cycle},  $ d_P(u_1)+d_P(u_q)\leq k+1$, and so $d_P(u_1), d_P(u_q) \in \{k/2, k/2 + 1\}$. 

%We will show that again that 
%\begin{equation}\label{u1_2}
%\mbox{$u_1$ has a neighbor outside of $P$.}
%\end{equation}

We will show that $u_1$ has a neighbor in $G'$ outside of $P$. The result for $u_q$ follows by symmetry.

{\em Case 2.2.1: $|C| = 2k-4$}. By Lemma~\ref{cycle}, $d_P(u_1) = d_P(u_q) = k/2$. As before, the vertex $u_2$ is in $G'$ and hence has at least $k/2 + 1$ neighbors in $G'$. Suppose first that all neighbors of $u_2$ in $G'$ are in $P$. As in the previous case, $u_2$ cannot have a neighbor that is a successor of a neighbor of $u_q$. If $u_2$ has a neighbor $u_r$ with $r \in \{i_1+ 1, \ldots, i_2-5\}$ then $u_2u_r \cup P[u_r, u_q] \cup u_qu_{i_1} \cup P[u_{i_1}, u_2]$ is a cycle with length at least $|C| - 1 + 5 \geq 2k$, a contradiction. 
%Also, if $u_{i_2 - 1} \in N(u_2)$, let $u_\ell$ be the first neighbor of $u_q$ on $P$. Then $u_{\ell-2} \in N(u_1)$ since $u_{\ell -3} \notin N_P(u_q)^+$. Therefore the cycle \[u_2u_{i_2-1} \cup P[u_{i_2-1}, u_\ell] \cup u_\ell u_q \cup P[u_q, u_{i_2}] \cup u_{i_2} u_1 \cup u_1 u_{\ell-2} \cup P[u_{\ell -2},u_2]\] contains $|V(P)|-1$ vertices, a contradiction.

Thus every neighbor of $u_2$ in $G'$ is  in $N(u_1)^- + u_{{i_2} - 3}$. Also by symmetry $N_P(u_{q-1}) \subseteq N(u_q)^+ + u_{i_1 + 3}$, so $N_P(u_2)$ and $N_P(u_{q-1})$ intersect in at most one vertex. As before, $u_2$ and $u_{q-1}$ cannot share a neighbor outside of $P$. Therefore 
\begin{equation}\label{bigX}|X^*| \geq |N_{G'}(u_2) \cup N_{G'}(u_{q-1})| \geq (k/2 + 1) + (k/2 + 1) - 1 \geq k+1.
\end{equation}
If $u_{i_2-1} \in N(u_2)$, then the cycle $P[u_2, u_{i_1}] \cup u_{i_1} u_q \cup P[u_q, u_{i_2 - 1}] \cup u_{i_2-1}u_2$ contradicts Claim~\ref{u1cycle}. So $d_P(u_2) \leq |N(u_1)^- + u_{i_2-3} - u_{i_2-1}| = k/2$, hence $u_2$ has a neighbor $u_1'$ in $G'$ outside of $P$. 
%%If $u_1'$ has a neighbor outside of $P$, then we consider instead the path $P' = u_1' \cup P[u_2, u_q]$ and obtain our result. Otherwise, all neighbors of $u_1'$ are in $P$. If $u_{i_2 - 2} \in N(u_1')$, then the cycle $C'=u_1' u_{i_2 - 2} \cup P'[u_{i_2 - 2}, u_q] \cup u_{q}u_{i_1} \cup P'[u_{i-1}, u_1']$ has length $|C| + 2 = 2k-2$, and we instead consider $C'$ and move onto the case where $|C'| = 2k-2$.
%
%Similarly, if $u_1'$ has any neighbors within $\{u_{i_1+1}, \ldots u_{i_2-4}\}$ then we obtain a cycle of length $2k$ or longer. Thus again by~\eqref{j152} and since $d(u_1') \geq d(u_1)$, we must have $N(u_1') = N(u_1)$. By~\eqref{bigX}, we may apply Lemma~\ref{ind} to complete the case.
If $u_1'$ is adjacent to some vertex $u_r$ with $r \in \{i_1+2, \ldots, i_2\}$, then the cycle $P'[u_1', u_{i_1}] \cup u_{i_1}u_q \cup P'[u_q, u_r] \cup u_ru_1'$ contradicts Claim~\ref{u1cycle}. So by~\eqref{j152}, $N_P(u_1') \subseteq N(u_1) - u_{i_2}$, and hence $d_P(u_1') \leq k/2 - 1 < d(u_1')$, so $u_1'$ has a neighbor outside of $P$. Then again we consider $P'$ and $u_1'$, and complete the case. 

{\em Case 2.2.2: $|C| = 2k-2$}. Assume $N_{G'}(u_1)$(and $N_{G'}(u_q)$) $\subseteq V(P)$. Suppose first that $d_P(u_1) = k/2$. 
%If $C \cap A = P^-(N(u_1)) \cup P^+(N(u_q))$, then a similar argument as in Case 1 shows that either $u_1$ has a neighbor outside of $P$ or $|X^*| \leq \frac{k}{2r-k+1}(n-1+D(G, X^*))$.
As before, we will show that $u_2$ must have a neighbor $u_1'$ in $G'$ outside of $P$. So suppose first that $u_2$ has no such neighbors.

If it exists, let $u_{i_0}$ be the unique vertex in $C \cap A$ which is not contained in $N(u_1)^- \cup N(u_q)^+$. Note that in this case, $C \cap A$ contains all vertices in $N(u_1)^- \cup N(u_q)^+ - u_{i_1+1} - u_{i_2-1} + u_{i_0}$. In particular, $|N(u_1)^- \cup N(u_q)^+ - u_{i_1+1} - u_{i_2-1} + u_{i_0}| \leq k-1$ if and only if $d(u_1) = d(u_q) = k/2$.  

If $u_2$ is adjacent to a vertex $u_r$ with $r \in \{i_1 + 1, \ldots, i_2 - 3\}$, then the cycle $P[u_2, u_{i_1}] \cup u_{i_1}u_q \cup P[u_q, u_{r}] \cup u_ru_2$ has at least $|C| + 2 \geq 2k$ vertices. Therefore $N_{G'}(u_2) \subseteq N(u_1)^- + u_{i_0}$. We obtain a similar result for $u_{q-1}$ by symmetry, and again we get $|X^*| \geq k+1$.

% If equality holds, then in particular $u_{i_2 - 1} \in N(u_2)$. Let $u_\ell$ be the first neighbor of $u_q$ in $P$. If $u_{\ell - 2} \in N(u_1)$, then the cycle $u_2u_{i_2-1} \cup P[u_{i_2-1}, u_\ell] \cup u_\ell u_q \cup P[u_q, u_{i_2}] \cup u_{i_2} u_1 \cup u_1 u_{\ell-2} \cup P[u_{\ell -2},u_2]$ has length at least $|V(P)| - 1 \geq 2k$. 
%
%Otherwise, the cycle $C = P

%
%
%If $u_2$ has a neighbor $u_r$ with $i_1 < r< i_2$ then as in the previous case, set $u_\ell$ to be the first neighbor of $u_q$ in $P$. Then either \[u_2u_{i_2-1} \cup P[u_{i_2-1}, u_\ell] \cup u_\ell u_q \cup P[u_q, u_{i_2}] \cup u_{i_2} u_1 \cup u_1 u_{\ell-2} \cup P[u_{\ell -2},u_2]\] or \[P[u_2, u_{i_1}] \cup u_{i_1}u_q \cup P[u_q, u_r] \cup u_ru_2\] has length at least $\min\{|V(P)| - 1, |C| + 2\}$, a contradiction. So we obtain $N(u_2) \subseteq N(u_1)^+ - u_{i_2-1} + u_{i_0}$. But $| N(u_1)^+ - u_{i_2-1} + u_{i_0}| = k/2 < d_{G'}(u_2)$  (again we note that $u_{i_0}$ exists only if $d(u_1) = k/2$). 

We also have that $u_2u_{i_2 - 1} \notin E(G)$, otherwise the cycle $P[u_2, u_{i_1}] \cup u_{i_1} u_q \cup P[u_1, u_{i_2-1}] \cup u_{i_2-1} u_2$ contradicts Claim~\ref{u1cycle}. Therefore $d_P(u_2) \leq k/2$. This implies $u_2$ has a neighbor $u_1'$ in $G'$ outside of $P$.
If $u_1'$ has a neighbor outside of $P$, then we instead consider the path $P' = u_1' \cup P[u_2, u_q]$ and are done. 
%If $u_1'$ has a neighbor in $u_r$ with $r \in \{i_1+1, \ldots, i_2-1\}$, then the cycle $P'[u_1', u_{i_1}] \cup u_{i_1}u_q \cup P'[u_q, u_r] \cup u_ru_1'$ is longer than $C$. Therefore every neighbor of $u_1'$ is in $C$, but $u_1'$ is not. As in the proof of Claim~\ref{u2uq-1}, we have that $G - u_1'$ is 2-connected, a contradiction to Lemma~\ref{2connind} (applied to $u_1'$ and $P'$). 
%Next suppose $d_P(u_1) = k/2+1$. In this case we have that each $u_i\in C\cap A$ is in $P^-(N_P(u_1))\cup P^+(N_P(u_q))$ and $|C| = 2k-2$. As in the previous case, $u_2$ cannot have a neighbor $u_r$ where $i_1<r<i_2$. Therefore $N_P(u_2) \subseteq P^-(N(u_1)) - u_{i_2 - 1}$. Since $d(u_1) \geq k/2 + 1 >  |P^-(N(u_1)) - u_{i_2 - 1}|$, we have that $u_2$ has a neighbor $u_1'$ in $G'$ outside of $P$. Then by the case, $N(u_1') \subseteq N(u_1)$, unless $u_1'$ has neighbors outside of $P$ (in which case we obtain~\eqref{u1_2}). Thus applying Lemma~\ref{ind} instead to $P'=u_1' \cup P[u_2, u_q]$ and $u_1'$, we obtain a contradiction.
As in the previous case, $u_1'$ does not have neighbors in $\{u_{i_1+2}, \ldots, u_{i_2}\}$. Hence $N(u_1')$ is contained in $C$ but $u_1'$ is not, contradicting Claim~\ref{u1cycle} (applied to $P'$ and $u_1'$).
 
This completes the proof for $u_1$. By symmetry, we have that $u_p$ also contains a neighbor outside of $P$.
\end{proof}

 For $j=1$ and $j=q$, let $z_j\in N(v_j) - V(P)$.
Since $G$ is $2$-connected, this implies that there is a path $Q_j=w_{1,j},\ldots,w_{\ell_j,j}$  from $u_j$ through $z_j=w_{2,j}$ to $P-u_j$ internally disjoint from $P$.
 Let $w_{\ell,j}=u_{h_j}$. If $Q_1$ and $Q_q$ share a vertex outside of $P$, then $G$ has a cycle containing $P$, a contradiction to $c(G)\leq 2k-2$. So, the only vertex
  common for $Q_1$ and $Q_q$ could be $u_{h_1}$ if it coincides with $u_{h_q}$. Also, $h_1\geq 3$, since if $h_1=2$, then the path $Q_1\cup P[u_2,u_q]$ is longer than $P$.
  Similarly, $h_q\leq q-2$.
 
  We claim that
\begin{equation}\label{j151}
 \mbox{for  $h_1-4\leq g\leq h_1-1$,  $u_{g}\notin N(u_q)$ and for  $h_q+1\leq g\leq h_q+4$,  $u_{g}\notin N(u_1)$ .}
\end{equation}
Indeed, by symmetry suppose $u_gu_q\in E(G)$ for some  $h_1-4\leq g\leq h_1-1$. Then  the cycle $P[u_1,u_g]\cup u_gu_q\cup P[u_q,u_{h_1}]\cup Q_1$ would have at least
$$2k+1-(h_1-g-1)+(j_1-1)\geq 2k+1-3+1=2k-1$$ vertices. This contradicts  $c(G)\leq 2k-2$.

Also
\begin{equation}\label{j152'}
 \mbox{$\{u_{h_1-2},u_{h_1-1},u_{h_q+1},u_{h_q+2}\}\cap  B=\emptyset$.}
\end{equation}
Indeed if $h_1-2\leq g\leq h_1-1$ and $u_g\in B$, then the path
$P[u_{g},u_1]\cup Q_1\cup P[u_{h_1},u_q]$ starts from $u_{g}\in B$  and is longer than
$P$ (because if $g=h_1-2$ then by parity, $\alpha_1\geq 4$). The proof for 
 $h_q+1\leq g\leq h_q+2$ is symmetric.

Similarly to~\eqref{j97}, we show
\begin{equation}\label{j154}
 \parbox{14.5cm}{\em (i) if $|C|=2k-4$, then $u_{h_1}\notin P[u_{i_1+1},u_{i_2-2}]$ and
$u_{h_q}\notin P[u_{i_1+2},u_{i_2-1}]$;\\
 (ii) if $|C|=2k-2$, then $u_{h_1}\notin P[u_{i_1+1},u_{i_2}]$ and
$u_{h_q}\notin P[u_{i_1},u_{i_2-1}]$
.}
\end{equation}
Indeed, if for example,  $i_1+1\leq h_1\leq i_2-2$, then the cycle $P[u_1,u_{i_1}]\cup u_{i_1}u_q\cup P[u_q,u_{h_1}]\cup Q_1$ would have at least
$|C|+3$ vertices, which means at least $2k$ vertices. All other possibilities are very similar.

 Let $\lambda$ be the odd integer
in the set $ \{h_1-3,h_1-2\}$.
 Similarly, let 
 $\mu$ be the  odd integer in the set $ \{h_q+2,h_q+3\}$.
 By~\eqref{j151}, $u_{\lambda}\notin  P^+(N_P(u_q))$.
By~\eqref{j152}, we have the following cases.

First suppose $ |C|= 2k-4$ and each $u_i\in C\cap A$ is in $N_P(u_1)^-\cup N_P(u_q)^+$. 
Since  $N_P(u_1)^-\cap N_P(u_q)^+=\emptyset$ and each of $u_1$ and $u_q$ has $k/2$ neighbors in $B$, this means
\begin{equation}\label{j156}
 \mbox{\em all neighbors of $u_1$ and $u_q$ in $P$ are in $B$.}
\end{equation}
By~\eqref{j152'} and~\eqref{j156},  $u_{\lambda}\notin  N_P(u_1)^-$. So by the case and the fact that $\lambda$ is odd, 
$i_1+1\leq \lambda\leq i_2-1$. This means $i_1+3\leq h_1\leq i_2+2$. Since $u_{i_2}\in B$, by~\eqref{j154}(i) and~\eqref{j152'},
$i_2-1\leq h_1\leq i_2$. Similarly, $i_1\leq h_q\leq i_1+1$.
 Then the cycle
$$P[u_1,u_{h_q}]\cup Q_{q}\cup P[u_q,u_{h_1}]\cup Q_1$$
has length at least $|C|+4$, contradicting $c(G)\leq 2k-2$.

Next, suppose  $ |C|= 2k-2$ and there is exactly one $i_0$ such that $u_{i_0}\in (C\cap A)-(N_P(u_1)^-\cup N_P(u_q)^+)$.
As the case $|C| = 2k-4$ this yields~\eqref{j156}. By~\eqref{j152'} and~\eqref{j156},  $u_{\lambda}\notin  N_P(u_1)^-$. So by the case and the fact that $\lambda$ is odd, 
either $\lambda=i_0$, or $i_1+1\leq \lambda\leq i_2-1$. If the latter holds, then $i_1+3\leq h_1\leq i_2+2$, which is impossible by~\eqref{j154}(ii) and~\eqref{j152'}.
Thus $\lambda=i_0$. Similarly, we conclude $\mu=i_0$. In particular, $h_q<h_1$.
 Since $\lambda=\mu$ is odd, the cycle
$$P[u_1,u_{h_q}]\cup Q_{q}\cup P[u_q,u_{h_1}]\cup Q_1$$
has length at least $|V(P)|-1\geq 2k$, contradicting $c(G)\leq 2k-2$.

Finally, suppose $ |C|= 2k-2$ and each $u_i\in C\cap A$ is in $N_P(u_1)^-\cup N_P(u_q)^+$. 
By Lemma~\ref{cycle}, $d_P(u_1)+d_P(u_q)\leq k+1$. So by the symmetry between $u_1$ and $u_q$, we may assume
$d_P(u_1)=k/2$ and hence $N_P(u_1)=N_{G'}(u_1)$. Since $i_0$ does not exist,
 repeating the argument of Case 2.2.2,  we get a contradiction even earlier.
\end{proof}
}

\section{Proof of Theorem~\ref{EGbgr} for general graphs}

%Everywhere in this section, we assume that positive integers $r,k$ are fixed with $r \geq k+1$, $k\geq 4$.

%
%Recall that a bigraph $G = (X, Y; E)$ is called $2k$-saturated if $c(G) < 2k$ but for any $x \in X, y \in Y$, if $xy \notin E(G)$, then $G + xy$ contains a cycle of length $2k$ or longer.

%We need some definitions. A bigraph $G=(X,Y;E)$
% is
%$2k$-{\em saturated} if $c(G)<2k$, but for each $x\in X$ and $y\in Y$ with $xy\notin E(G)$, the graph $G+xy$ has a cycle of length at least $2k$.
%For example, if $s\leq k-1$ then for any $t$ graph $K_{s,t}$ is $2k$-saturated, because it does not have
%$x\in X$ and $y\in Y$ with $xy\notin E(G)$.
%  A bigraph $G=(X,Y;E)$
% is
%$2k$-{\em block-saturated} if every block of $G$  is $2k$-saturated.

\smallskip
For disjoint vertex sets $X$ and $Y$, an $(X,Y)$-{\em frame} is a pair $(G,X^*)$ where $G$ is a bigraph with parts $X$ and $Y$, and $X^*\subseteq X$.

A block $G'$ in an $(X,Y)$-{frame}  $(G,X^*)$ with parts $X'$ and $Y'$ is {\em special} if all of the following holds:\\
(i) $G'=K_{k-1,r}$ with $|X'|=k-1$;\\
(ii) $X'\subseteq X^*$;\\
(iii) $N_G(x)=Y'$ for each $x\in X'$.

Let $Q(G,X^*)$ denote the number of {special} blocks in an $(X,Y)$-{frame}  $(G,X^*)$. Recall the definition of deficiency:

\smallskip
%If $x$ is a vertex in an $(X,Y)$-{frame}  $(G,X^*)$, then let $D(x)=\max\{0, r-d_G(x)\}$. Let
\[
D(G,X^*) = \sum_{x\in X^*}D_G(x)=\sum_{x\in X^*}\max\{0, r-d_G(x)\}.
\]
The following theorem implies Theorem~\ref{EGbgr}.

\begin{thm}\label{t1}
Let $ k\geq 4$, $r\geq k+1$  and $m,m^*,n$ be positive integers with $m^*\leq m$. %Let $k$ {\bf be even}.
Let $(G,X^*)$ be an $(X,Y)$-{frame},
where $|X|=m$, $|Y|=n$, and
 $|X^*|=m^*$, and $G$ is $2k$-saturated. %For every $x\in X$, let $D(x)=\max\{0, r-d(x)\}$. Let $D=D(G)=\sum_{x\in X^*}D(x)$.
If $c(G)<2k$, then 
\begin{equation}\label{dc1}
m^*\leq \frac{k-1}{r}(n-1+D(G,X^*)+Q(G,X^*)).
\end{equation}

Furthermore, equality holds if and only if $G$ and $X^*$ satisfy the following: \\
(i) $G$ is connected;\\
(ii) all blocks of $G$ are copies of either $K_{k-1, r}$ or $K_{k-1, r+1}$ with the partite set of size $k-1$ in $X$ and all cut vertices of $G$ in $Y$;\\
(iii) $X^* = X$;\\
(iv) $D(G, X^*) = 0$.
\end{thm}

It is straightforward to check that the graphs described in (i)-(iv) are indeed sharpness examples to Theorem~\ref{t1}: suppose $G$ has $s$ blocks of the form $K_{k-1, r}$ and $t$ of the form $K_{k-1, r+1}$. Then $m = m^* = (s+t)(k-1)$, $n = s(r-1) + tr + 1$, $D(G, X) = 0$, and $Q(G, X) = s$, since each $K_{k-1, r}$ block is special. Therefore 
\[\frac{k-1}{r}(n-1 + D(G, X) + Q(G, X)) = \frac{k-1}{r}(s(r-1) + tr + 1 - 1 + 0 + s) = \frac{k-1}{r}(r(s + t)) = m^*.\]

{\bf Proof of Theorem~\ref{t1}.} Let $(G,X^*)$ be a counterexample to the theorem with the fewest vertices in $G$. 
For short, let $D=D(G,X^*)$ and $Q=Q(G,X^*)$.
 By the definition of $D(x)$,
\begin{equation}\label{dc3}
 d(x)+D(x)\geq r\mbox{ for every ${x\in X}$.}
\end{equation}

{\bf Case 1:} $G$ is $2$-connected. If $m^*\geq k-1$ and $m\geq k$ , then~(\ref{dc1}) follows Theorem~\ref{main}. In fact, we get strict inequality as $\frac{k}{2r-k+2} < \frac{k-1}{r}$ whenever $k-1 < r$.
Suppose $1\leq m^*\leq k-2$ and $x\in X^*$. Then by~(\ref{dc3}), 
$$n-1+D\geq d(x)+D(x)\geq r-1,$$
so $\frac{k-1}{r}(n-1 + D) \geq \frac{k-1}{r}(r-1) > k-2 \geq m$.

The last possibility is that $m^*=m=k-1$. If $n+D\geq r+1$, then~(\ref{dc1}) holds, so suppose 
$n+D=r$. Since $k-1\geq 2$, this together with~(\ref{dc3}), implies that $n=r$ and $D=0$. 
Thus $G=K_{k-1,r}$ and $X^*=X$ which yields that $G$ is a special block.
Thus $Q=1$ and so $n-1+D+Q=r$. This finishes Case 1.

Note that equality is obtained only in this subcase where $G = K_{k-1, r}, X = X^*$, and $D = 0$. Therefore $G$ and $X^*$ satisfy (i)-(iv). 

\medskip
Since Case 1 does not hold, $G$ has a pendant block, say with vertex set $B$. Let $b$ be the cut vertex in $B$,
$X^*_B=X^*\cap B-b$,
$m^*_B=|X^*_B|$, and  $n_B=|B\cap Y|$. Furthermore, let $G_1=G-(B-b)$ and $n_1=|Y\cap V(G_1)|$. 

Note that $G_1$ is 2k-saturated: as a cycle cannot span multiple blocks in a graph, if there exists an edge $xy \notin E(G_1)$ such that $G_1 + xy$ contains no cycle of length $2k$ or longer, then $G + xy$ also contains no cycle of length $2k$ or longer, contradicting that $G$ is $2k$-saturated.

{\bf Case 2:} $b\in Y$. Let  $X^*_1=X^*-X^*_B$. By the minimality of $G$,
\begin{eqnarray}\label{yind}
 |X^*_1|\leq  \frac{k-1}{r}(n_1-1+D(G_1,X_1^*)+Q(G_1,X_1^*)), \text{ and } \label{Y1} \\ 
  m^*_B\leq  \frac{k-1}{r}(n_B-1+D(G[B],X_B^*)+Q(G[B],X_B^*)) \label{Y2},
\end{eqnarray}
using $m^*= |X^*_1|+ m^*_B$, we obtain
\begin{equation}\label{dc5}
m^* \leq  \frac{k-1}{r}\Big(n_1+n_B-2+D(G_1,X_1^*)+D(G[B],X_B^*)
+Q(G_1,X_1^*)+Q(G[B],X_B^*)\Big).
\end{equation}
Since $n_1+n_B-2=n-1$, $D=D(G_1,X_1^*)+D(G[B],X_B^*)$ and 
$Q=Q(G_1,X_1^*)+Q(G[B],X_B^*)$,~(\ref{dc5}) implies~(\ref{dc1}).

Furthermore, if equality holds in \eqref{dc1}, then we  have equalities in both \eqref{Y1} and \eqref{Y2}. Again by the minimality of $G$, frames $B$ with $X^*_B$ and $G_1$ with $X^*_1$  both satisfy (i)-(iv). In particular, we have $X^*_B = X \cap B$ and $X^*_1 = X - B$. Since $X^* = X^*_B \cup X^*_1 = X$, it follows  that $G$ also satisfies (i)-(iv). 

{\bf Case 3:} $b\in X$ and $X^*_B=\emptyset$. By the minimality of $G$,
\begin{equation}\label{dc6}
 |X^*|\leq  \frac{k-1}{r}(n_1-1+D(G_1,X^*)+Q(G_1,X^*)).
\end{equation}
Since $d_G(b)-d_{G_1}(b)\leq n_B$, $D(G_1, X^*) \leq D(G, X^*) + n_B$. If $Q(G_1,X^*)=Q(G,X^*)$, then~(\ref{dc6}) implies~(\ref{dc1}). Furthermore, suppose that equality holds in \eqref{dc1}. Then equality also holds in \eqref{dc6}, and $D(G, X^*) + n_B= D(G_1, X^*)$. By the minimality of $G$, $G_1$ and $X^*$ satisfy (i)-(iv). In particular by  (iv), $D(G_1, X^*) = 0$, contradicting that $D(G, X^*) + n_B= D(G_1, X^*)$.

So, suppose $Q(G_1,X^*)>Q(G,X^*)$. By Part (iii) of the definition of a special block, if this happens, then
$Q(G_1,X^*)=Q(G,X^*)+1$ and the unique block $B_1$ that is special in $(G_1,X^*)$  but not special in 
$(G,X^*)$ contains $b$. This means $d_{G_1}(b)=r$ and hence $D(G_1,X^*)=D$. But $n_B\geq 1$, and so
again~(\ref{dc6}) implies~(\ref{dc1}). Furthermore, if $n_B \geq 2$, then we obtain strict inequality in \eqref{dc1}. If $n_B = 1$, say $Y \cap B = \{y\}$, then since $B$ is 2-connected, $B$ consists of a single edge $yb$ attached to the special block $B_1$, where $B_1$ is a copy of $K_{k-1, r}$ with $|X \cap B_1| = k-1$. Note that if a block in $G$ has a partite set of size $k-1$, then the longest cycle in $G$ has length at most $2(k-1)$. Thus for any $x \in X \cap (B_1 - b)$, $c(G + xy) \leq 2(k-1)$, contradicting that $G$ is $2k$-saturated. 

{\bf Case 4:} $b\in X$ and $1\leq m^*_B\leq k-2$. 
Let  $X^*_1=X^*-X^*_B-b$.
By the minimality of $G$,
\begin{equation}\label{dc7}
 |X^*_1|<  \frac{k-1}{r}(n_1-1+D(G_1,X_1^*)+Q(G_1,X_1^*)).
\end{equation}
Where note that we have strict inequality because $X^*_1$ does not satisfy (iii) for $G_1$. Since $b\notin X^*_1$, $Q(G_1,X_1^*)=Q$. Let $x\in X^*_B$. By~(\ref{dc3}),
$D(x)+n_B\geq r$. Thus by the case,
$$m^* \leq |X^*_B\cup \{b\}|+|X^*_1|<
k-1+\frac{k-1}{r}(n_1-1+D(G_1,X_1^*)+Q)\leq 
$$
$$
k-1+\frac{k-1}{r}(n-n_B-1+(D-D(x))+Q)\leq k-1+\frac{k-1}{r}(n-1+D-r+Q)=\frac{k-1}{r}(n-1+D+Q),
$$
as claimed.

{\bf Case 5:} $b\in X$ and $ m^*_B\geq k-1$. 
Let  $X^*_1=X^*-X^*_B-b$.
By the minimality of $G$, again~(\ref{dc7}) holds. Also, as in Case 4, $Q(G_1,X_1^*)=Q$ and $m^*\leq |X^*_1|+1+m^*_B$.
Since $n=n_1+n_B$, and $ D(G_1,X_1^*)+D(G[B],X^*_B) \leq D$, in order for $(G,X^*)$ to be a counterexample
to the theorem, all this together yields
\begin{equation}\label{dc8}
 m^*_B+1>  \frac{k-1}{r}(n_B+D(G[B],X_B^*)).
\end{equation}

On the other hand, by Theorem~\ref{t1},
$$m^*_B\leq \frac{k}{2r-k+2}(n_B-1+D(G[B],X^*_B)).
$$
Plugging this into~(\ref{dc8}), we get
\begin{equation}\label{dc9}
 \frac{k}{2r-k+2}(n_B-1+D(G[B],X^*_B))+1>\frac{k-1}{r}(n_B+D(G[B],X_B^*)).
 \end{equation}
 Since the coefficient at $n_B+D(G[B],X^*_B)$ in the left side of~(\ref{dc9}) is less than the one in the right side, and since by~(\ref{dc3}),  
 $n_B+D(G[B],X^*_B)\geq r$,~(\ref{dc9}) implies
 \begin{equation}\label{dc10}
 \frac{k}{2r-k+2}(r-1)+1>\frac{k-1}{r}r=k-1.
 \end{equation} 
 But~(\ref{dc10}) is equivalent to $k(r-1)>(k-2)(2r-k+2)$, which is not true when $r\geq k\geq 4$.
 \qed

As a corollary, we obtain the same result for graphs that are not $2k$-saturated.

 \begin{cor}\label{t2}
Let $ k\geq 4$, $r\geq k+1$  and $m,m^*,n$ be positive integers with $m^*\leq m$. %Let $k$ {\bf be even}.
Let $(G,X^*)$ be an $(X,Y)$-{frame},
where $|X|=m$, $|Y|=n$, and
 $|X^*|=m^*$. %For every $x\in X$, let $D(x)=\max\{0, r-d(x)\}$. Let $D=D(G)=\sum_{x\in X^*}D(x)$.
If $c(G)<2k$, then 
\begin{equation}\label{dc2}
m^*\leq \frac{k-1}{r}(n-1+D(G,X^*)+Q(G,X^*)).
\end{equation}
 \end{cor}
\begin{proof}
Add edges to $G$ until the resulting graph is $2k$-saturated. Call this graph $G'$. Note that when adding edges, the deficiency $D(G, X^*)$ of any $X^* \subseteq X$ cannot grow. That is, $D(G,X^*) \leq D(G', X^*)$. 

Applying Theorem~\ref{t1} to $G'$, we obtain
\[
|X^*| \leq \frac{k-1}{r}(n-1+D(G',X^*)+Q(G',X^*)). \]

If $Q(G',X^*) \leq Q(G, X^*)$, then we're done. Otherwise suppose that $Q(G', X^*) = Q(G,X^*) + t$. This implies that there were $t$ special blocks created when adding edges within the blocks of $G$. Let $B$ be such a block that was not special in $G$ but became special in $G'$. Then in $G$, $B \subsetneq K_{k-1,r}$. Thus some vertex $v \in B \cap X$ has $d_G(v) < r$ but $d_{G'}(v) = r$. Hence $D_G(v) > D_{G'}(v)$. It follows $D(G, X^*) \geq D(G', X^*) + t$. 

Thus \[|X^*|  \leq \frac{k-1}{r}(n-1 + D(G', X^*) + Q(G, X^*) + t) \leq \frac{k-1}{r}(n-1 + D(G, X^*) + Q(G, X^*),\] as desired. 
\end{proof} 

\section{Proofs for hypergraphs:  Theorem~\ref{EG_full} and Corollary~\ref{gyori2} }

{\bf Proof of Theorem~\ref{EG_full}}. Let $\mathcal H$ be an $n$-vertex multi-hypergraph with lower rank $r$ and edge multiplicty at most $k-2$. Let $G = G(\mathcal H)$ be the incidence graph of $\mathcal H$ with parts $X = V(\mathcal H)$ and $Y = E(\mathcal H)$. By construction, since $\mathcal H$ has lower rank at least $r$, each $x \in X$ has $d_G(x) \geq r$. Therefore $D(G, X) = 0$. Also, $G$ cannot contain a special block (i.e., $Q(G, X) = 0$) as such a block in $G$ would correspond to a set of $k-1$ edges in $\mathcal H$ that are composed of the same $r$ vertices. But we assumed that $\mathcal H$ has no edges with multiplicity greater than $k-2$.

%First suppose $r =k$. By the choice of $\mathcal H$ as an edge-minimal counterexample, we have that $\mathcal H$ (and therefore $G$) is connected. If $G$ is 2-connected, then we apply Theorem~\ref{main} to $G$ with $X^* = X$ to obtain $m^* \leq \max\{n-k+1, \frac{k}{2r-k+2}(n-1)\} \leq n-1$. So we may suppose that $G$ has a pendant block $B$ and cut vertex $b$ such that $V(B) \cap V(G') = \{b\}$ where $G' = G - B + b$. If $b \in Y$, let $\mathcal B$ and $\mathcal H'$ be the subhypergraphs of $\mathcal H$ corresponding to $B$ and $G'$ respectively. By the choice of $\mathcal H$, $\mathcal B$ and $\mathcal H'$ satisfy Theorem~\ref{EG_full}. I.e., we obtain $|\mathcal H| = |\mathcal B| + |\mathcal H'| \leq |V(\mathcal B)| - 1 + |V(\mathcal H')| - 1 =n-1$ where we use the fact that $|V(\mathcal B)| + |V(\mathcal H')| = n-1$. If $b \in X$, then $b$ corresponds to a hyperedge in $\mathcal H$. let $\mathcal C_1$ be a component of $\mathcal H - b$, and let $\mathcal H' = \mathcal H - b -  \mathcal C_1$. Again, by the choice of $\mathcal H$, $\mathcal C_1$ and $\mathcal H'$ satisfy Theorem~\ref{EG_full}. Therefore $|\mathcal H| = |\mathcal C_1| + |\mathcal H'| + 1 \leq |V(\mathcal C_1)| -1 + |V(\mathcal H')| - 1 + 1 = n-1$. 

 Applying Theorem~\ref{t1} to $G$ with $X^* = X$, we obtain \[e(\mathcal H) = |X| \leq \frac{k-1}{r}(n-1 + D(G, X) + Q(G, X)) = \frac{k-1}{r}(n-1).\]

Finally, suppose equality holds. Add edges to $G$ until it is $2k$-saturated. Let $G'$ be the resulting graph. Again we have $Q(G', X) = Q(G, X) = 0$ and $D(G', X) = D(G, X) = 0$, therefore $|X| = \frac{k-1}{r}(n-1 + D(G', X) + Q(G', X))$. Hence $G'$ satisfies (i)-(iv) in the second part of the statement of Theorem~\ref{t1}. In particular, all blocks of $G'$ are copies of $K_{k-1,r+1}$ with cut vertices in $Y$. 
%Then in $G$, we must have that within each block, every vertex $x \in X$ has a unique set of $r$ neighbors within the $r+1$-partite set. That is, each $K_{k-1,r+1}$ block in $G'$ corresponds to an $(r,k)$-block in the hypergraph $\mathcal H$. This completes the proof of Theorem~\ref{EG}.
Then in $G$  within each block, every vertex $x \in X$ is adjacent to a subset of the $r+1$-partite set of size $r$ or $r+1$. That is, each $K_{k-1, r+1}$ block in $G'$ corresponds to an $(r+1,k-1)$-block in $\mathcal H$. This completes the proof of Theorem~\ref{EG_full}. \qed

{\bf Proof of Corollary~\ref{gyori2}}. 
 Recall that a Berge-path of length $k$ has $k+1$ base vertices and $k$ hyperedges.
Suppose $\mathcal H$ satisfies the conditions of the corollary.
We construct the multi-hypergraph $\mathcal H'$ by adding a new vertex $x$ to $\mathcal H$ and extending each hyperedge of $\mathcal H$ to include $x$. Then $\mathcal H'$ has $n+1$ vertices, lower rank at least $r+1$,  no edge with multiplicity at least $k-1$, and $e(\mathcal H') = e(\mathcal H)$. 

We claim that $\mathcal H'$ has no Berge-cycle of length $k$ or longer. Suppose there exists such a cycle with edges $e_1, \ldots, e_\ell$ and base vertices $v_1, \ldots v_\ell$ and $\ell \geq k$. If $x \in \{v_1, \ldots, v_k\}$, say $x = v_1$, then since each edge in $\mathcal H'$ contains at least $r+1$ vertices, there exist distinct vertices $v_1' \in e_1 - \{v_1, \ldots, v_{k}\}$ and  $v_{k+1}' \in e_k - \{v_1, \ldots, v_{k}\}$. For each $1 \leq i \leq \ell$, let $e_i' = e_i - \{x\}$. Then $e_1', \ldots, e_k'$ and $\{v_1', v_2, \ldots, v_{k}, v_{k+1}'\}$ form a Berge-path of length $k$. The case where $x \notin \{v_1, \ldots, v_k\}$ is similar (and simpler). Therefore, applying Theorem~\ref{EG} to $\mathcal H'$, we obtain \[e(\mathcal H) = e(\mathcal H') \leq \frac{k-1}{r + 1}((n+1) - 1),\] as desired. \qed

\vspace{4mm}
{\bf Acknowledgment.} We  thank Zolt\`an F\"uredi and Jacques Verstra\"ete for very helpful discussions.
We also thank Ervin Gy\H ori for sharing a proof of Theorem 1.7.

\end{document}